\documentclass[11pt]{amsart}

\usepackage[T1]{fontenc}
\usepackage[utf8]{inputenc}
\usepackage[mathscr]{eucal}
\usepackage{dsfont}
\usepackage{textcomp}

\usepackage{amssymb, amsfonts, amsthm}
\usepackage{mathtools} 

\usepackage[dvipsnames]{xcolor} 
\usepackage{tikz}
\usetikzlibrary{positioning}

\usepackage[a4paper, twoside=false, vmargin={2cm,3cm}, includehead]{geometry}
\usepackage{comment}
\usepackage{mdwlist}
\usepackage{enumitem} 

\usepackage{hyperref}
\hypersetup{
    colorlinks=true,
    linktocpage=true,
    linkcolor=red,
    filecolor=blue,
    citecolor=blue,
    urlcolor=cyan,
}

\newtheorem{lemma}{Lemma}[section]
\newtheorem{theorem}[lemma]{Theorem}
\newtheorem*{theorem*}{Theorem}
\newtheorem{corollary}[lemma]{Corollary}
\newtheorem{question}{Question}
\newtheorem{proposition}[lemma]{Proposition}
\newtheorem*{proposition*}{Proposition}

\newtheorem*{problem*}{Problem}

\theoremstyle{definition}

\newtheorem*{claim*}{Claim}
\newtheorem*{convention}{Convention}

\newtheorem{definition}[lemma]{Definition}

\newtheorem*{remark}{Remark}

\DeclareMathOperator*{\E}{\mathbb{E}}

\newcommand{\C}{{\mathbb C}}
\newcommand{\D}{{\mathbb D}}

\newcommand{\N}{{\mathbb N}}
\newcommand{\NN}{{\mathcal N}}

\newcommand{\Q}{{\mathbb Q}}
\newcommand{\R}{{\mathbb R}}
\renewcommand{\S}{\mathbb{S}}

\newcommand{\Z}{{\mathbb Z}}
\newcommand{\U}{{\mathbb U}}

 \newcommand{\OK}{{\mathcal{O}_K}}

\newcommand{\mf}{\mathfrak}

\newcommand{\e}{\epsilon}

\newcommand{\norm}[1]{\left\Vert #1\right\Vert}

\newcommand{\FM}[1]{{\color{violet}{{#1}}}}
\newcommand{\WS}[1]{{\color{blue}{{#1}}}}

\newcommand{\AK}[1]{{\color{ForestGreen}{{#1}}}}

\author{Sebasti\'an Donoso, Andreu Ferr\'e Moragues, Andreas Koutsogiannis, and Wenbo Sun}

\address[Sebasti{\'a}n Donoso]{Departamento de Ingenier\'{\i}a Matem\'atica and Centro de Modelamiento Matem{\'a}tico, Universidad de Chile \& IRL 2807 - CNRS, Beauchef 851, Santiago, Chile} \email{sdonoso@dim.uchile.cl}
\address[Andreu Ferr\'e Moragues]{Department of Mathematical Analysis and Applied Mathematics, Faculty of Mathematics, Complutense University of Madrid, 28040 Madrid, Spain
}
\email{anferre@ucm.es}

\address[Andreas Koutsogiannis]{
Department of Mathematics, Aristotle University of Thessaloniki, Thessaloniki 54124, Greece}
\email{akoutsogiannis@math.auth.gr}

\address[Wenbo Sun]{Department of Mathematics, Virginia Tech, 225 Stanger Street, Blacksburg, VA, 24061, USA}
\email{swenbo@vt.edu}

\thanks{The first author was partially funded by ANID/Fondecyt/1241346 and Centro de Modelamiento Matemático (CMM) FB210005, BASAL funds for centers of excellence from ANID-Chile. This project was implemented in the framework “3rd Call for H.F.R.I.’s Research Projects to Support Faculty Members \& Researchers” (H.F.R.I. Project Number: 24979). The fourth author was supported by the NSF Grant DMS-2247331.}


\title{Partition regularity in imaginary quadratic rings of integers}

\begin{document}
\begin{abstract}

We obtain partition regularity results for homogeneous quadratic equations whose parametrized solutions admit nice factorizations into linear forms over rings of integers of imaginary quadratic fields. To do so, we develop number-theoretic results of independent interest on such fields, such as  a characterization for aperiodic completely multiplicative functions, the Tur\'an-Kubilius inequality, and a new concentration estimate for multiplicative functions.

\end{abstract}
\maketitle
\tableofcontents

\section{Introduction}\label{s1}
\subsection{Partition and density regularity}
A central problem in Ramsey theory is to determine when the existence of solutions of a polynomial equation is preserved under finitely many partitions. For example, let $q(x,y,z)=0$ be a polynomial equation with 3 variables. Is it true that for any finite coloring of the domain $D$ of the polynomial $q$, there exist $x,y,z$ of the same color with $q(x,y,z)=0$? We say that the equation $q(x,y,z)=0$ is \emph{partition regular over $D$} if the answer to this question is affirmative. A famous example due to Schur \cite{Schur} is that the linear equation $x+y=z$ is partition regular over $\N$. This result was extended by Rado in \cite{Rado} where he characterized all the linear equations that are partition regular.

On the other hand, the partition regularity problem for non-linear equations is significantly harder. In \cite{Gr07} and \cite{Gr08}, Erd\H{o}s and Graham asked whether the Pythagorean equation $x^{2}+y^{2}=z^{2}$ is partition regular over $\N$. 
Although this question remains open, significant progress has been made in the last decade. Specifically, for the Pythagorean equation and variations of it, one can find pairs from the set $\{x,y,z\}$ that are of the same color. We formalize this relaxation of the problem with the following definition. 

\begin{definition}
    Let $D \subset \C$ be a set and $q\colon D^{3}\to \C$ be a function. We say that $q(x,y,z)=0$ is \emph{partition regular over $D$ in $(x,y)$} if for every finite coloring of $D$, there exist distinct nonzero $x,y,z\in D$ with $x,y$ having the same color  such that $q(x,y,z)=0$. Partition regularity over $D$ in  $(x,z)$ and $(y,z)$ is defined analogously.
\end{definition} 
Pioneer work achieving quadratic partition regularity results includes that of S\'ark\"{o}zy in \cite{sarkozy_1978a}, who showed that the equation $x-y=z^{2}$ is partition regular over $\N$ in $(x,y)$, as well as the work of Khalfalah and Szemer\'{e}di (see \cite{KS06}), who showed that the equation $x+y=z^{2}$ is partition regular over $\N$ in  $(x,y)$. 

Using a decomposition result for multiplicative functions in terms of Gowers norms (introduced in \cite{Go01}, and which have been successfully employed in the study of patterns and very general partition regularity results, e.g., \cite{GT08b}, \cite{BMR20}, \cite{Fr21}, \cite{Fr22}, \cite{FHo17}, among many others), Frantzikinakis and Host showed in \cite{FHo17} that given $a,b,c\in\Z$, if $\sqrt{-ac}, \sqrt{-bc}, \sqrt{-(a+b)c}\in\Z$, then the equation $ax^{2}+by^{2}+cz^{2}=0$ is partition regular over $\N$ in $(x,y)$ (such an example is the equation $9x^{2}+16y^{2}-z^{2}=0$). 
This result was extended in \cite{S18, S23}, where it was shown that if $\sqrt{-ac}, \sqrt{-bc}, \sqrt{-(a+b)c}$ belong to some number field $K$, then
the equation $ax^{2}+by^{2}+cz^{2}=0$ is partition regular over $\mathcal{O}_{K}$ (the ring of integers of $K$) in $(x,y)$  (for example $x^{2}+y^{2}-z^{2}=0$ is partition regular over $\Z[\sqrt{2}]$ in $(x,y)$, and is partition regular over $\Z[i]$ in $(x,z)$ and $(y,z)$).

The result of \cite{FHo17} was improved recently by Frantzikinakis, Klurman and Moreira \cite{FKM1} where the original assumption on the square roots was weakened to $\sqrt{-ac}, \sqrt{-bc}\in\Z$. This was further improved to $\sqrt{-ac}$ or $\sqrt{-bc}\in\Z$ by the same authors in \cite{FKM2}. As a special case, they showed that $x^{2}+y^{2}-z^{2}=0$ is partition regular over $\N$ in any two of the variables. One of the key innovations of \cite{FKM2,FKM1} is that the authors replaced the usage of a decomposition result for multiplicative functions introduced in \cite{FHo17} by a convenient splitting of the set of multiplicative functions into aperiodic and non-aperiodic, and then studied each of the components separately (as opposed to the previous approach, which decomposed all functions simultaneously).
Indeed, using this new methodology, it was shown in \cite{FKM1} that $x^{2}+y^{2}-z^{2}=0$ is also partition regular over $\N$ in $(x,z)$ and $(y,z)$. It is important to highlight that, in order to do so, a quadratic concentration estimate for the non-aperiodic functions was developed and used. 

In this paper, we extend the splitting approach in \cite{FKM2,FKM1} in the setting of imaginary quadratic number fields (i.e., $K=\Q(\sqrt{-d})$ for some squarefree $d\in\N$) and  provide novel applications to various partition regularity problems, extending many results from \cite{FHo17,FKM2,FKM1,S18,S23}. 
  In this setting, the ring of integers $\OK$ of $K$ is given by $\OK=\Z[\tau_{d}]$, where $\tau_{d}=\sqrt{-d}$ if $d \equiv 2,3 \pmod 4$, and $\tau_d=\frac{1+\sqrt{-d}}{2}$ if $d \equiv 1 \pmod 4$.
  The next result is an extension we obtain of the aforementioned partition result \cite[Corollary~1.3]{FKM1} to imaginary quadratic number fields.


\begin{theorem}\label{th1}
    Let  $d\in\N$ be squarefree. For $a,b,c\in\Z\backslash\{0\}$, if $\sqrt{-ac}, \sqrt{-bc}\in \Z[\tau_{d}]$, then the equation $ax^{2}+by^{2}+cz^{2}=0$ is partition regular over $\Z[\tau_{d}]$ in $(x,y)$.
\end{theorem}

In fact, we obtain a density version of Theorem \ref{th1}.
Recall that a \emph{multiplicative F\o lner sequence in} $\OK$ is a sequence of finite subsets of $\mathcal{O}_K^{\times}$ such that
\[
\lim_{M \to \infty} \frac{|\Phi_M \cap (x^{-1} \cdot \Phi_M)|}{|\Phi_M|} = 1, \text{ for all } x \in \mathcal{O}_K^{\times}.
\]
The precise F\o lner sequences that we will use (cf. \eqref{folnerMult}) are a bit intricate to define, so for now, the reader can assume that F\o lner sequences look like the following example when $\OK=\Z$, which contains all the main ingredients for the generalization that appears in \eqref{folnerMult}:
\[
\Phi_M:=\{p_1^{a_1}\dots p_k^{a_k} : M \leq a_i \leq 2M, 1 \leq i \leq M\}.
\]
where $(p_i)_{i\in \N}$ are the prime numbers. 

For each F\o lner sequence we have a notion of largeness for subsets $\Lambda \subseteq R^{\times}$, defined via
\begin{equation}\label{eq: largeness}
\overline{d}_{\Phi}(\Lambda):=\limsup_{M \to \infty} \frac{|\Lambda \cap \Phi_M|}{|\Phi_M|},
\end{equation}
where the upper bar is dropped if the limit does actually exist. We say that $\Lambda \subseteq \mathcal{O}_K^{\times}$ has positive \emph{multiplicative upper Banach density} if there exists some multiplicative F\o lner sequence $\Phi=(\Phi_M)_{M}$ along which we have $\overline{d}_{\Phi}(\Lambda)>0$. 


\begin{definition}[Density regularity]\label{def: density reg}
    Let $q\colon D^{3}\to \C$ be a function. We say that $q(x,y,z)=0$ is \emph{density regular over $D$ in $(x,y)$} if for every subset $\Lambda \subseteq D$ with positive multiplicative upper Banach density, there exist distinct nonzero $x,y,z\in D$ with $x,y\in \Lambda$ such that $q(x,y,z)=0$. Density regularity over $D$ in $(x,z)$ and $(y,z)$ is defined analogously.
\end{definition}

We have the following result, which is stronger than (and easily implies) Theorem \ref{th1}, given the subadditivity of the density we just introduced.

\begin{theorem}\label{th1density}
    Let  $d\in\N$ be squarefree. For $a,b,c\in\Z\backslash\{0\}$, if $\sqrt{-ac}, \sqrt{-bc}\in \Z[\tau_{d}]$, then the equation $ax^{2}+by^{2}+cz^{2}=0$ is density regular over $\Z[\tau_{d}]$ in $(x,y)$.
\end{theorem}

It is important to note that more standard additive versions of density will not be helpful to deal with the equation $x^2+y^2=z^2$, since if, for example, we look at $\Lambda=\{ u \in \Z[i]^{\times} : u \equiv 1 \pmod{2}\}$, then having $x, y \in \Lambda$, implies that $x^2+y^2$ can never be a square (we can still look at it modulo $4$ and argue similarly as in the case of $\N$).
It turns out that the more convenient form of density will be one that is multiplicative instead of additive, given the nature of the quadratic equations under consideration.


We remark that Theorems \ref{th1} and \ref{th1density} improve the results in \cite{S18,S23} by dropping the requirement that $\sqrt{-(a+b)c}\in \Z[\tau_{d}]$. 
For example, Theorem \ref{th1} implies that if $d\in\N$ is squarefree, then the equation $x^{2}+dy^{2}-z^{2}=0$ is partition regular over the ring of integers of $\Z[\tau_d]$. This result was previously known only over the ring of integers coming from the number field $\Q(\sqrt{-d},\sqrt{-(d+1)})$ (see \cite{S23}).

Our method applies to not only quadratic forms of the form $ax^{2}+by^{2}+cz^{2}$, but to any quadratic forms $q(x,y,z)=0$ for which the parametrized solutions for $x$ and $y$ factorize linearly in  $\Z[\tau_{d}]$. We refer the reader to Theorem \ref{thmain} for details.

It is natural to ask whether Theorems \ref{th1} and \ref{th1density} hold when $d<0$, or whether they hold when only one of $\sqrt{-ac}, \sqrt{-bc}$ belongs to $\Z[\tau_{d}]$.
The discussion of these natural questions is postponed to Section~\ref{SEC: further}.

\subsection{Number-Theoretic tools} In order to prove the above-mentioned partition regularity results, we obtain a series of results in analytic number theory for imaginary quadratic number fields. To the best of our knowledge, these results were previously unknown, and many of them require substantial new ideas compared to the classical results over integers. We believe these results are of independent interest, and may have potential applications in other number-theoretic problems.

Here is a summary of the number-theoretic tools we develop in this paper. We postpone their precise statements to the next Section.

\subsubsection{Characterizations for completely multiplicative functions}
 It is a classical result of Daboussi and Delange (see \cite[Corollary~1]{daboussidelange82}) that a completely multiplicative function $f\colon \Z^{\times}\to\U$ is either \emph{aperiodic}, meaning that its average along every arithmetic progression is zero, or \emph{pretentious}, meaning that it behaves similarly to the product of a Dirichlet character and an Archimedean character. 
This is a crucial tool for the splitting method introduced in \cite{FKM2,FKM1}.
In this paper, we obtain an analog of \cite[Corollary~1]{daboussidelange82} in the setting of imaginary quadratic number fields (i.e., Theorem \ref{characterization}) by providing a list of equivalent definitions of aperiodic completely multiplicative functions (see Proposition \ref{prop: dircharequivalence}). 
An interesting feature is that to define a pretentious multiplicative function, one needs to consider its extension to ideal domained functions. We refer the readers to Section \ref{ss5} for details.

\subsubsection{The Tur\'an-Kubilius inequality}
Another important step in proving the partition regularity results of this paper is obtaining an analog of \cite[Proposition~2.5]{FKM1} (namely, Theorem \ref{p25}), which is a linear concentration estimate that allows us to easily treat the pretentious multiplicative functions. As concentration estimates are based on the Tur\'an-Kubilius inequality, in order to prove  Theorem \ref{p25}, we obtain an analog of the inequality for imaginary quadratic fields in Theorem \ref{TK}.
We refer the readers to Section \ref{SECconcest} for all the details.

\subsection{Organization of the paper}
We state all the main results in density regularity and number theory in Section \ref{spps}.
Then in Section~\ref{sec: proofstrat} we explain the strategy  for the proof  of the main density regularity result (i.e., Theorem \ref{thmain}) 
and break it down into several positivity results. A lot of the key reductions that help properly frame the problem are also explained in detail.

Next, in Section~\ref{SEC: numthybackg} we provide some necessary number theoretical background. 
We then prove the characterization result for completely multiplicative functions in Section \ref{ss5}, and the Tur\'an-Kubilius inequality as well as the main concentration estimate we need in Section \ref{SECconcest} respectively.
Finally, in  Section~\ref{sec: proofmain} we combine all  the tools developed up to that point to prove the positivity results that are needed to complete the proof of Theorem \ref{thmain}.


We finish the paper with some discussion on possible future avenues in Section~\ref{SEC: further}. 
Lastly, the Appendix is used to prove a result that allows us to reconcile ball averages with box averages.
\\ \\
\noindent {\bf{Acknowledgements.}} We thank Diego C\'espedes for bringing the paper \cite{lucht2001} to our attention. This simplified a previous version of our work, in which we had proved a version of Hal\'asz's theorem from scratch, a result already established in \cite{lucht2001}.

\section{Precise statements of the main results}\label{spps}
In this section we will introduce the relevant technical definitions so that the main results can be precisely stated.
\subsection{The main density regularity result and applications}

As was stated in the Introduction, Theorem \ref{th1density} can be generalized to any quadratic equation $q(x,y,z)=0$ for which the parametrized solutions for $x$ and $y$ factorizes linearly in  $\Z[\tau_{d}]$. In order to make this more precise, we introduce the following definition.
 
\begin{definition}
    Let $K$ be a number field and $q\colon\mathcal{O}_{K}^{3}\to\C$ be a map. We say that $q$ admits an \emph{$\mathcal{O}_{K}$-factorization in $(x,y)$} if there exist some map $P\colon \mathcal{O}_K^{2}\to \OK$ and maps $L_{i}\colon \mathcal{O}_{K}^{2}\to \OK, 1\leq i\leq 4,$ of the form $L_{i}(m,n)=a_{i,1}m+a_{i,2}n$ for some $a_{i,1},a_{i,2}\in\mathcal{O}_{K}$ with $L_{1}, L_2, L_3, L_{4}$ being pairwise linearly independent such that for all $m,n,k\in \mathcal{O}_{K}$ the following is a solution to the equation $q(x,y,z)=0$:
    $$x=kL_{1}(m,n)L_{2}(m,n), y=kL_{3}(m,n)L_{4}(m,n) \text{ and } z=kP(m,n).$$
    We say that $p$ admits a \emph{strong $\mathcal{O}_{K}$-factorization in $(x,y)$} if we may further require $a_{1,1}=a_{2,1}=a_{3,1}=a_{4,1}=1$.
\end{definition}

We will proceed by proving the following result which implies Theorems \ref{th1} and ~\ref{th1density}. 
 \begin{theorem}\label{thmain}
     Let $K=\Q(\sqrt{-d})$ with $d\in\N$ squarefree and $q\colon\mathcal{O}_{K}^{3}\to\C$ be a map admitting an $\mathcal{O}_{K}$-factorization in $(x,y)$.
     Then $q(x,y,z)=0$ is density regular over $\OK$ in $(x,y)$.    
 \end{theorem}

It was essentially proved in \cite{S23} that Theorem \ref{thmain} holds for all number fields $K$ if $q$ admits a strong $\mathcal{O}_{K}$-factorization in $(x,y)$. In the present paper, we relax this condition by $\mathcal{O}_{K}$-factorization.
\begin{proof}[Proof of Theorems \ref{th1} and ~\ref{th1density} assuming Theorem~\ref{thmain}]
    Note that the solutions of the equation $ax^{2}+by^{2}+cz^{2}=0$ admit the following parametrization:
    $$x=k\sqrt{-bc^{3}}(m-n)(m+n),\;\; y=2k\sqrt{-ac^{3}}mn,\;\; z=k\sqrt{abc^{2}}(m^{2}+n^{2}),$$
    where $\sqrt{-ac^{3}},\sqrt{-bc^{3}}\in\Z[\tau_{d}]$ and $\sqrt{abc^{2}}=\sqrt{-ac}\cdot\sqrt{-bc}\in\Z[\tau_{d}]$.
   The conclusion follows from Theorem \ref{thmain} and the Pigeonhole Principle.
\end{proof}

Next we present some applications of Theorem \ref{thmain}.
In \cite{FKM2,FKM1}, it was proven that $x^{2}+y^{2}-z^{2}=0$ is partition regular over $\Z$ in any of the three pairs of variables. 
Using Theorem \ref{th1}, in the case where $d=1$, i.e., the Gaussian integers, we obtain examples where partition regularity holds in any of the three pairs of variables.

\begin{corollary}\label{cor: gaussian integers nice}
    The equation $ax^{2}+by^{2}+cz^{2}=0$ is partition regular over $\Z[i]$ in $(x,y)$, $(x,z)$ and $(y,z)$ if $\sqrt{-ac}, \sqrt{-bc},\sqrt{-ab}\in \Z[i]$ (or equivalently if $\sqrt{a}, \sqrt{b},\sqrt{c}\in \Z[i]$).
\end{corollary}
 
For example, we have that the equation $x^{2}+y^{2}+z^{2}=0$ is partition regular over $\Z[i]$ in any of the three pairs of variables. (Note that this equation cannot be partition regular over $\Z$.)

\begin{remark}
Full partition regularity is impossible for every generalized Pythagorean equation (of the form  $ax^2+by^2+cz^2=0$) with squares over a ring of integers $\OK$. Indeed, if we hope to have full partition regularity of the equation $ax^2+bx^2=cz^2$ over $\OK$, then, at the very least, we will have full partition regularity for the equation $ax+by=cz$ over $\OK$. Rado proved in \cite{rado43} an extension of his result over $\Z$ for subrings of $\C$ (in particular, any $\OK$ fits the bill), which implies that what is now known as Rado's condition is necessary. Whether it is sufficient is wide open for any $\OK$, including, of course, the case where $\OK=\Z$.
\end{remark}




   
 It is natural to ask for examples for which partition regularity holds in any of the three pairs of variables over other number fields. To do so we need to include terms $xy,yz, zx$ (or we can only hope to obtain a similar result in the Gaussian integers; see the remark after Theorem~\ref{th2density} below). Consider the more general quadratic polynomials of the form
\begin{equation}\label{fp}
    q(x,y,z)=ax^{2}+by^{2}+cz^{2}+exy+fxz+gyz
\end{equation}
for some $a,b,c,e,f,g\in\Z$. Denote 
$$\Delta_{1}(q):=f^{2}-4ac,\; \Delta_{2}(q):=g^{2}-4bc,\;\text{and}\; \Delta_{3}(q):=(f+g)^{2}-4(a+b+e)c.$$
It was shown in \cite[Proposition~10.4]{S23} that $q(x,y,z)=0$ is partition regular in $(x,y)$ over $\mathcal{O}_{K}$, where $K=\Q(\sqrt{\Delta_{1}(q)},\sqrt{\Delta_{2}(q)},\sqrt{\Delta_{3}(q)})$. We remark that
\cite[Proposition~10.4]{S23} can be recovered using Theorem~\ref{thmain}, which we will prove below. In addition,
  we show that the equation is partition regular with some other choices of number fields, leading to new applications.
 Let $$\Delta_{4}(q)=c(ce^{2}+bf^{2}+ag^{2}-efg-4abc).$$
We have the following.

\begin{theorem}\label{th2density}
    Let $q$ be given by \eqref{fp} and $d\in\N$ be squarefree. 
     Suppose that $a,b,c\neq 0$ and $\Delta_{4}(q)\neq 0$.\footnote{If $\Delta_{4}(q)=0$, then the proof of Theorem \ref{th2density} implies that $q$ can be factorized into the product of two linear equations with $\Q(\sqrt{-d})$-coefficients, and the problem is reduced to the linear case.} Then the equation $q(x,y,z)=0$ is  
\begin{itemize}
    \item density regular in $(x,z)$ if $\sqrt{\Delta_{2}(q)},\sqrt{\Delta_{4}(q)}\in\Z[\tau_{d}]$;
    \item density regular in $(y,z)$ if $\sqrt{\Delta_{1}(q)},\sqrt{\Delta_{4}(q)}\in\Z[\tau_{d}]$; and
    \item density regular in $(x,y)$ if $\sqrt{\Delta_{1}(q)},\sqrt{\Delta_{2}(q)},\sqrt{\Delta_{4}(q)}\in\Z[\tau_{d}]$.
\end{itemize}
\end{theorem}

 As an application of Theorem \ref{th2density}, we have that the equation
$$x^{2}+y^{2}+3z^{2}+xy=0$$
is partition regular over $\Z[\sqrt{-3}]$ in any of the three pairs of variables, which is a result that cannot be obtained with the methods of \cite{S23}. 

\begin{proof}[Proof of Theorem \ref{th2density} assuming Theorem \ref{thmain}]

     We first consider the case when $f=g=0$. In this case, $\Delta_{1}(q), \Delta_{2}(q)$ and $\Delta_{4}(q)$ are equal to $-4ac$, $-4bc$ and $c^2(e^{2}-4ab)$ respectively. 
    We may rewrite $q(x,y,z)=0$ as 
    $$a(x-\lambda_{-}y)(x-\lambda_{+}y)+cz^{2}=0$$
    where
    $$\lambda_{\pm}=\frac{-e\pm\sqrt{e^{2}-4ab}}{2a}.$$
    Since $\Delta_{4}(q)\neq 0$, we have $\lambda_{+}\neq \lambda_{-}$. Since $a,b\neq 0$, we have $\lambda_{\pm}\neq 0$.
    So we have a solution if 
    $$x-\lambda_{-}y=kcm^{2},\;\; x-\lambda_{+}y=-kan^{2}, \;\; \text{and}\;\; z=kamn,$$
    or equivalently,
     $$x=\frac{k}{\lambda_{+}-\lambda_{-}}(\lambda_{-}cm^{2}+\lambda_{+}an^{2}),\;\; y=\frac{k}{\lambda_{+}-\lambda_{-}}(cm^{2}+an^{2}),\;\;  \text{and}\;\; z=kamn.$$

     Note that $\lambda_{+}-\lambda_{-}=\frac{\sqrt{e^{2}-4ab}}{a}=\frac{\sqrt{\Delta_{4}(q)}}{a}$, $\sqrt{-ac}=\frac{1}{1}\sqrt{\Delta_{2}(q)}$ and  $\sqrt{-\lambda_{+}\lambda_{-}ac}=\sqrt{-bc}=\frac{1}{2}\sqrt{\Delta_{2}(q)}$. It follows that $q$ admits an 
     \begin{itemize}
    \item $\OK$-factorization in $(x,z)$ if $\sqrt{\Delta_{2}(q)},\sqrt{\Delta_{4}(q)}\in\Z[\tau_{d}]$.
    \item $\OK$-factorization in $(y,z)$ if $\sqrt{\Delta_{1}(q)},\sqrt{\Delta_{4}(q)}\in\Z[\tau_{d}]$.
    \item $\OK$-factorization in $(x,y)$ if $\sqrt{\Delta_{1}(q)},\sqrt{\Delta_{2}(q)},\sqrt{\Delta_{4}(q)}\in\Z[\tau_{d}]$.
\end{itemize}
 We are done by Theorem \ref{thmain}.

    We now consider the general case.
    Let 
    \begin{eqnarray*}
        p'(x,y,z) & = & p(2cx,2cy,z-fx-gy) \\
        & = &-c\left(\Delta_{1}(q)x^{2}+\Delta_{2}(q)y^{2}-z^{2}+(\Delta_{3}(q)-\Delta_{1}(q)-\Delta_{2}(q))xy\right).
    \end{eqnarray*}
    From the previous case and the identity 
    $$4\Delta_{4}(q)=(\Delta_{3}(q)-\Delta_{1}(q)-\Delta_{2}(q))^{2}-4\Delta_{1}(q)\Delta_{2}(q),$$
    we have the conclusion.
\end{proof}
\begin{remark}
It is interesting to note that if one insists on having a diagonal equation (of the form $ax^2+by^2=cz^2$) and wants to obtain full partition regularity for any of the three possible pairs using Theorem~\ref{thmain}, then after some checking which is left to the interested reader, one can show that, necessarily, the quadratic field extension must be $\Z[i]$, the Gaussian integers, which highlights the special properties that this ring of integers enjoys.
\end{remark}

 \subsection{Characterizations for completely multiplicative functions}


The strategy of the proofs for the partition regularity results of this paper is to replace the decomposition method used in \cite{S18,S23} with the splitting method introduced in \cite{FKM2,FKM1}. 
%
Recall that a function $f\colon \N\to\U$, where $\U:=\{z \in\C: |z| \leq 1\}$ is \emph{multiplicative} if $f(mn)=f(m)\cdot f(n)$ whenever $(m,n)=1$, and is \emph{completely multiplicative} if $f(mn)=f(m)\cdot f(n)$ for all $m,n\in\N$.
For functions $f,g: \N\to\U$, define
$$\D(f,g):=\sum_{p \text{ is a prime number in } \N} \frac{1}{p}(1-\text{Re}f(p)\overline{g}(p)).$$
We say that $\chi : \Z \to \U$ is a \emph{Dirichlet character of modulus $m$} if $\chi$ is completely multiplicative, satisfies $\chi(a)=0$ if and only if $\gcd(a,m)>1$, and is periodic with period $m$. It is a classical result (see \cite[Corollary~1]{daboussidelange82}) that a completely multiplicative function $f\colon \Z^{\times}\to\U$ is either \emph{aperiodic}, meaning that
$$\lim_{N\to\infty}\frac{1}{N}\sum_{n=1}^{N}f(a+bn)=0$$
for all $a,b\in\N$, or \emph{pretentious}, meaning that $\D(f\chi,n^{i\tau})<\infty$ for some Dirichlet character $\chi$ and $\tau\in\R$.


In this paper, we develop an analog of this characterization for completely multiplicative functions over imaginary quadratic number fields.
  We say that $f\colon \mathcal{O}_K^{\times}\to\U$ is \emph{completely multiplicative} if $f(mn)=f(m)\cdot f(n)$ for all $m,n\in \mathcal{O}_K^{\times}$. Let $d\in\N$ be squarefree and $\NN\colon \Q(\sqrt{-d})\to \Q$ be the norm of the field extension (see Section~\ref{SEC: numthybackg} for more details). 
We define the notion of aperiodic functions in this context in the following way.

\begin{definition}[Aperiodic functions]\label{aperiodicdef}
   Let $d\in\N$ be squarefree and let $\mathcal{AP}[\tau_{d}]$ denote the collection  of all  2-dimensional grids in $\Z[\tau_{d}]$, i.e., the collection of sets of the form
   \begin{equation}\label{eq:adcd} 
\{(a_{1}m+b_{1})+(a_{2}n+b_{2})\tau_{d}\in \Z[\tau_{d}]\colon m,n\in\Z\}
\end{equation}
        for some $a_{1},b_{1},a_{2},b_{2}\in\Z$ with $a_{1},a_{2}\in\N$.
   We say that a function $f\colon\Z[\tau_{d}]^{\times}\to\C$ is \emph{aperiodic} if the limit
\begin{equation}\label{eq: aperiodicdef}
\lim_{N\to\infty}\E_{u\in P,\NN(u) \leq N}f(u)
\end{equation}
exists and equals to 0 for all $P\in\mathcal{AP}[\tau_{d}]$. 
\end{definition}

\begin{remark}
    We remark this notion of aperiodicity is only well defined when $d>0$, as the set $\{n\in\Z[\tau_{d}]\colon \NN(n)\leq N\}$ is infinite when $d<0$.
We also caution the reader that Definition~\ref{aperiodicdef} is different from the one that appeared in \cite{S23}, where the averages are taken over boxes instead of balls. 
It is an interesting question to ask if these two definitions are equivalent. While aperiodic functions defined in both ways share similar properties (for example, they admit similar seminorm estimates for multiple ergodic averages, and their Gowers norms both vanish, see Appendix \ref{appapp}), in this paper it is crucial that we work with ball averages. We refer the readers to Section \ref{s:as} for details on the discussion regarding the choice of averging scheme.

\end{remark}

It is then natural to ask if such a classification result holds for completely multiplicative functions over other number fields. While such a result seems feasible for principal ideal domains (or equivalently, a unique factorization domain, since $\Z[\tau_d]$ is a Dedekind domain), for non-principal ideal domains it is a harder task to achieve. In the latter, we need to consider the extensions of the function into ideal-domained functions.
%

Let $\mathfrak{I}(K)$ be the set of all ideals of $\mathcal{O}_K^{\times}$. We say that a map $g\colon \mathfrak{I}(K)\to \U$ is \emph{multiplicative} if $g(IJ)=g(I)\cdot g(J)$ for all coprime  $I,J\in\mathfrak{I}(K)$,\footnote{See Section \ref{SEC: numthybackg} for definition.} and is  \emph{completely multiplicative} if $g(IJ)=g(I)\cdot g(J)$ for all  $I,J\in\mathfrak{I}(K)$.

This allows us to define extensions of completely multiplicative functions.
\begin{definition}\label{def: multfunction extension}
Let $f\colon \mathcal{O}_K^{\times}\to\U$ be a completely multiplicative function with $f(\e)=1$ for all units $\e$. We say that a map $\tilde{f}\colon \mathfrak{I}(K)\to \U$ is an \emph{extension} of $f$ if $\tilde{f}$ is completely multiplicative and $\tilde{f}((n))=f(n)$ for all $n\in \mathcal{O}_K^{\times}.$\footnote{Since $(\e n)=(n)$ for all units $\e$ and $n\in\mathcal{O}_K^{\times}$, $f$ admits an extension only if $f(\e)=1$ for all units $\e$.} 
\end{definition}
For  $f,g\colon\mathfrak{I}(K)\to\C$, define
\begin{equation}\label{ddf}
    \D(f,g):=\sum_{\mathfrak{p} \text{ is a prime ideal of } \mathcal{O}_{K}}\frac{1}{\NN(\mathfrak{p})}(1-\text{Re}f(\mathfrak{p})\overline{g}(\mathfrak{p})).
\end{equation}
It follows from Proposition~\ref{prop: extensions} (whose proof is deferred) that, when $K=\Q(\sqrt{-d}), d\in\N$, the number of extensions of every completely multiplicative function is equal to the ideal class number of $K$, and furthermore, they can be written down explicitly.

\begin{definition}[Dirichlet characters]
     Let $I$ be a non-trivial ideal of $\OK$. We say that a completely multiplicative function $\chi\colon \mathcal{O}_K^{\times}\to\C$ is a \emph{Dirichlet character} over $\OK$ with \emph{period $I$} if $\chi(x+y)=\chi(x)$ for all $x\in \OK, y\in I$ and if $\chi(x)=0$ if and only if $x \mod I\notin (\OK/I)^{\times}$.  
     Given a Dirichlet character $\chi$, we will call the function
     $\chi'\colon \mathcal{O}_K^{\times}\to\C$ a \emph{modified Dirichlet character} over $\OK$ with \emph{period $I$} if $\chi'(x)=\chi(x)$ for $x \mod I\notin (\OK/I)^{\times}$ and $\chi'(x)=1$ otherwise.  
\end{definition}
It is clear from the definition, but important in the sequel, that modified Dirichlet characters are $\S$-valued completely multiplicative functions.

\begin{definition}[Pretentious functions]
    We say that a function $f\colon\mathcal{O}_K^{\times}\to\S$  
  is \emph{pretentious} if $f(\e)=1$ for all units $\e$ and there exist some modified Dirichlet character $\chi$, some $\tau\in\R$, and some extension $\tilde{f'}$ of the function
$$f'(u):=f(u)\overline{\chi}(u)\NN(u)^{-i\tau}\footnote{Note that $f'$ is a multiplicative function taking values in $\S$ that is trivial on units.}$$
such that 
  $\D(\tilde{f'}, 1)<\infty$. We say that $f$ is \emph{super pretentious} if we can further require $\chi=1$ in the construction of $f'$.  
\end{definition}

 We have the following classification theorem.


\begin{theorem}\label{characterization}
     Let  $d\in\N$ be squarefree and $\NN\colon \Q(\sqrt{-d})\to \Q$ be the norm of the field extension. Let $f\colon \Z[\tau_{d}]^{\times}\to\S$ be completely multiplicative. Then either $f$ is aperiodic or  pretentious. 
\end{theorem}

\begin{remark}
Note that in Theorem \ref{characterization} we only classify $\S$-valued functions. We will study $\U$-valued functions in Section 4, but not in Theorem \ref{characterization}. The reason is that in Proposition \ref{prop: dircharequivalence}, we will only characterize $\S$-valued aperiodic functions, to ensure the existence of extensions of $f\chi$.
\end{remark}



As a consequence, we have the following corollary when the domain is a principal ideal domain (which we will shorten to PID). In order to state it we define a notion of distance between two multiplicative functions. We need to recall that $p \in \OK$ is a \emph{prime element} if it is not the product of two non-unit elements in $\OK$.\footnote{Recall that an element of $\OK$ is a unit if and only if its norm is equal to $\pm 1$.}

\begin{definition}
Let $f,g: \Z[\tau_d]\to \U$ be completely multiplicative functions. Then, the pretentious distance between $f$ and $g$ is given by
\[
\D(f,g):= \sum_{p \text{ is a prime element of } \Z[\tau_{d}]}\frac{1}{\NN(p)}(1-\text{Re}f(p)\overline{g}(p))
\]
\end{definition}
We now state a corollary of Theorem~\ref{characterization}.
\begin{corollary} 
     Let  $d\in\N$ be squarefree and $\NN\colon \Q(\sqrt{-d})\to \Q$ be the norm of the field extension. Let $f\colon \Z[\tau_{d}]^{\times}\to\S$ be completely multiplicative. If $\Z[\tau_{d}]$ is a PID, then either $f$ is aperiodic or  $\D(f\chi,\NN(n)^{i\tau})<\infty$ for some Dirichlet character $\chi$, and some $\tau\in\R$.
\end{corollary}

Theorem~\ref{characterization} is one of the major number theoretical inputs of the paper, which allows us to execute the machinery of \cite{FKM1} for imaginary quadratic number fields, especially when $\Z[\tau_{d}]$ is not a principal ideal domain. We believe that Theorem~\ref{characterization} is interesting in its own right, and may potentially have further applications in number theory.


The proof of Theorem~\ref{characterization} is based on an analog of Hal\'asz's theorem over number fields. The classical Hal\'asz's Theorem is a profound result in analytic number theory that characterizes the asymptotic behavior of the average of a multiplicative function (cf. \cite[Satz~1 and Satz~1$'$]{halasz68}). As it turns out, while this theorem can be relatively easily extended to PIDs such as $\Z[i]$ (see \cite[Thoerems A and 1.2]{dlms24}, which proved a special case of Hal\'asz's theorem over the Gaussian integers for real valued bounded completely multiplicative functions), difficulties arise when the elements in the domain cannot be factorized in a unique way. 
 Nevertheless, by passing to $\mf I(K)$-domained functions instead of $O_K^{\times}$ ones, the following extension of Hal\'asz's theorem is known to hold:

\begin{theorem}[Theorem 6.1, \cite{lucht2001}]\label{C22intro}
     Let $K$ be a number field and $g\colon\mf{I}(K)\to\U$ be a multiplicative function. Then, we have the following.
     \begin{enumerate}
         \item If $\inf_{\tau\in\R}\D(g,\NN(\mf u)^{i\tau})=\infty$, then $\lim_{x \to \infty} \E_{\NN(\mf u)\leq x}g(\mf u)=0$.
         \item If $\D(g,\NN(\mf u)^{i\tau})<\infty$ for some $\tau\in\R$, then
         $$\sum_{1 \leq \NN(\mf u)\leq x}g(\mf u) = \gamma_{K}\cdot\frac{x^{1+i\tau}}{1+i\tau}\prod_{\mf p\colon \text{prime ideal in } \mathcal{O}_K,\NN(\mf p)\leq x}(1-1/\NN(\mf p))(1+h_{1+i\tau}(\mf p))+o(x),$$
         where $h_{s}(\mf p):=\sum_{k=1}^{\infty}g(\mf p^{k})\NN(\mf p^{k})^{-s}$, and $\gamma_{K}$ 
         is the universal constant defined in Corollary \ref{lem: bounds for number of ideals of norm at most x}. 
     \end{enumerate}
\end{theorem}



We remark that Theorem \ref{C22intro} applies to any number field not just the quadratic imaginary ones. When $K=\Q$, Theorem \ref{C22intro} is the standard version of Hal\'asz's theorem.
Although Theorem \ref{C22intro} is only valid for $\mf I(K)$-domained functions, it turns out that this is good enough for the proof of Theorem~\ref{characterization} (for $\OK$-domained functions).

\subsection{The Tur\'an-Kubilius inequality}

The Tur\'an-Kubilius inequality is an important tool in analytic number theory (see for example \cite[Lemma~4.1]{E79book}). In this paper, we need an analog of it for additive functions over imaginary quadratic fields whose domains are ideals.
We say that $h: \mf I(K) \to \C$ is \emph{additive} if   $h(\mf u \mf v)=h(\mf u)+h(\mf v)$ for any coprime ideals $\mf u, \mf v$.
We have the following Tur\'an-Kubilius inequality for additive functions. 
It is worth pointing out that a different version of this result is proved in \cite[Chapter~10]{kubiliusbook} for $\Z[i]$.

\begin{theorem}[Tur\'an-Kubilius inequality]\label{TK}
    Let $K=\Q(\sqrt{-d})$ for some squarefree $d\in\N$, $h\colon\mf I(K)\to\C$ be an additive function, $a,Q\in\mathcal{O}_K^{\times}$ with $\NN(a)\leq \NN(Q)$ and $(a)$ coprime to $(Q)$, and also fix $M,N\in\N$. Put $N':=2\NN(Q)(N+1)$. Suppose that $Q\in \mf p$ for all prime ideal $\mf p$ with $\NN(\mf p)\leq M$.
    Then \begin{equation}\label{ltk}
    \begin{split}
       \E_{u\in\OK\colon\NN(u) \leq N}\left| h((Qu+a))-A_{Q,N}\right|^{2}
       \leq 
       (2+o_{Q;N\to\infty}(1))B_{N'}+O(C_{N}),
    \end{split}
\end{equation}
where the implicit constant depends only on $d$,
$$A_{Q,N}:=\sum_{\NN(\mf p^{k})\leq N,\mf p\nmid Q}h(\mf p^{k})\NN(\mf p)^{-k}(1-\NN(\mf p)^{-1}),$$
$$B_{N}:=\sum_{\NN(\mf p^{k})\leq N, \NN(\mf p)>M, k\geq 1}\vert h(\mf p^{k})\vert^{2}\NN(\mf p^{k})^{-1},$$
and
$$C_{N}:=\frac{1}{\sqrt{N}}\sum_{\NN(\mf p^{k})\leq N', \NN(\mf p)>M}\vert h(\mf p^{k})\vert^{2}\cdot \NN(\mf p^{k})^{-1/2}.$$
\end{theorem}

We remark that Theorem \ref{TK} differs from the classical  Tur\'an-Kubilius inequality in the following ways. The first is that in Theorem~\ref{TK} we deal with averages along arithmetic progressions $Qu+a$ instead of along intervals. The second is that in Theorem~\ref{TK} there is an extra error term $O(C_{K})$ which does not exist in the case when $K=\Q$. We do not know how to remove this term in Proposition~\ref{TK}. However, this term is harmless in our applications.

\subsection{Notation}
Given a finite set $A$, and a function $f$ on $A$ we use the expectation operator $\E$ to denote its averages:
\[
\E_{u \in A} f(u):=\frac{1}{|A|} \sum_{u \in A} f(u).
\]
For $N \in \N$ we will use the notation $[N]$ to mean the set of integers $\{1,\dots,N\}$.

Given $t \in \R$, we put $e(t):=e^{2\pi i t}$.
For $z \in \C$,
we use $\text{Re}(z)$ and $\text{Im}(z)$ for its real and imaginary parts respectively.   We also write $\exp(z)$ for its series $\sum_{n=0}^{\infty} z^n/n!$.  

We use standard number theory notation, so for example, $a(n) \ll b(n)$, or equivalently, $a(n)=O(b(n)),$ if there exists some constant $C>0$ such that $|a(n)| \leq C|b(n)|$ for large enough $n \in \N$. The little $o$ notation is also used: we write $a(n)=o(b(n))$ if $\lim_{n \to \infty} \frac{a(n)}{b(n)}=0$.

We use $\U$ to denote the closed unit ball in $\C$, so $\U=\{ z \in \C : |z| \leq 1\}$, and $\S^1$ for its boundary, so $\S^1=\{ z \in \C : |z|=1\}$.

We will typically fix a number field $K$, and $\OK$ will stand for its ring of integers. We will always use parentheses for the field and square brackets for the ring of integers, so typical notations will be $\Q(i)$ and $\Z[\sqrt{-2}]$, for example. Since we usually work with the non-zero elements of the ring of integers $\OK$, we often write $\mathcal{O}_K^{\times}$  to denote the set $\OK \setminus \{0\}$. 

To distinguish when we are working with elements $u,v \in \OK$ as opposed to ideals of $\OK$, we use in the latter gothic letters like $\mf u, \mf v$ or $\mf p$ (which we reserve for prime ideals). To denote the set of all ideals of $\OK$ we use $\mf I(K),$ taking only one representative modulo units.

We let $\NN: K \to K$ denote the norm associated to the number field $K$ (see the number theory section below for the relevant definitions). 

Most of our work applies to quadratic imaginary fields. In the case where $K$ is such a field, we will use $\tau_d$ to denote a generator of its ring of integers, so $\OK=\Z[\tau_d]$.

\section{An outline for the proof of the density regularity result}\label{sec: proofstrat}


In this section, we explain the outline of the proof of Theorem \ref{thmain}. For the rest of this Section \ref{sec: proofstrat}, we assume that $K=\Q(\sqrt{-d})$ with $d\in\N$ squarefree.

\subsection{Reductions}


In order to ease the notation for the proof of Theorem \ref{thmain}, we apply some standard reductions on the linear forms. First, by a change of variables, we may assume without loss of generality that $a_{4,2} \neq 0$.  The linear substitution $(m,n) \mapsto (a_{4,2}m,n-a_{4,1}m)$ shows that we may further assume that $a_{4,1}=0$ without loss of generality and without changing the assumptions on our linear forms. 

Next, it follows that $a_{1,1},a_{2,1},a_{3,1}$ are all non-zero so if we now make the change $n \mapsto a_{1,1}a_{2,1}a_{3,1} n$ and factor each $a_{i,1}$ from $L_i$ for $i=1,2,3$, we can further assume that $a_{1,1}=a_{2,1}=a_{3,1}=1$ modulo multiples by elements of $\mathcal{O}_K^{\times}$. Lastly, the change $m \mapsto m-a_{3,2}n$ allows us to reduce Theorem \ref{thmain} to the special case when
$$L_{1}(m,n)L_{2}(m,n)=\ell(m+\alpha n)(m+\beta n) \text{ and } L_{3}(m,n)L_{4}(m,n)=\ell'mn$$
for some $\ell, \ell' \in \mathcal{O}_K^{\times}$ and $\alpha,\beta \in \mathcal{O}_K$.

We first use a variation of the Furstenberg's correspondence principle, that applies to actions of $(\mathcal{O}_K^{\times}, \times)$ (see for example \cite[Theorem~2.8]{bf21b}). This allows us to translate Theorem~\ref{thmain} into the following.

\begin{theorem}\label{firstreductionplus}
Let $(T_u)_{u \in\mathcal{O}_K^{\times}}$ be a measure preserving multiplicative action of $(\mathcal{O}_K^{\times},\times)$ on a probability space $(X,\mathcal{B},\mu)$,\footnote{Meaning that $T_{u}$ is a measure preserving transformation for all $u \in\mathcal{O}_K^{\times}$ with $T_{1}=id$ and $T_{u}\circ T_{v}=T_{uv}$ for all $u,v\in\mathcal{O}_K^{\times}$.} and let $A \in \mathcal{B}$ with $\mu(A)>0$. Let $\ell,\ell' \in \mathcal{O}_K^{\times}$ and $\alpha, \beta \in \OK$. Suppose that $q: \mathcal{O}_K^3 \to \C$ admits an $\OK$-factorization in $(x,y)$. Then, there exists a set of positive lower density\footnote{Recall that a set $E \subset (\mathcal{O}_K^{\times})^2$ is said to have positive lower density if $\liminf_{N \to \infty} \frac{|E \cap B_N|}{|B_N|}>0$, where $B_N=\{ u \in \mathcal{O}_K^{\times} : \NN(u)\leq N\}$} of pairs $(m,n)\in (\mathcal{O}_K^{\times})^2$ so that the elements $\ell(m+\alpha n)(m+\beta n)$ and $\ell'mn$ are distinct, and
\begin{equation}\label{measurerec1plus}
\mu(T_{\ell(m+\alpha n)(m+\beta n)}^{-1}A \cap T_{\ell'mn}^{-1}A)>0.
\end{equation}
\end{theorem}

Our next step is to use the Herglotz-Bochner theorem to reduce Theorem~\ref{firstreductionplus} to a property on the spectrum of the system. Recalling the definition of completely multiplicative function on $\mathcal{O}_K^{\times}$ we introduced before, we denote, through the rest of the paper

\[
\mathcal{M}:= \{ f: \mathcal{O}_K^{\times} \to \S^1 : f \text{ is completely multiplicative}\}.
\]
We give $\mathcal{M}$ the product topology, which makes it into a compact metric space (given that $\S^1$ is). There is a natural identification between the Pontryagin dual of $(K^{\times},\times)$ and $\mathcal{M}$, which we shall use to more easily apply the Herglotz-Bochner theorem. (Indeed, this follows by taking an integral basis and taking common factor, so any term of the form $\frac{p}{q}x_1+\frac{r}{s}x_2$ is of the form $\frac{u}{v}$ for some $u, v \in \mathcal{O}_K^{\times}.$) 

Consider the map $\varphi: K^{\times} \to [0,1]$ given by $\varphi\left(\frac{u}{v}\right):=\mu(T_u^{-1}A \cap T_v^{-1}A)$, for $u, v \in \mathcal{O}_K^{\times}$. It is easy to check that $\varphi$ is well defined and positive definite. Thus, by the Bochner-Herglotz theorem, there exists a finite positive Borel measure $\sigma$ on $\mathcal{M}$ such that $\sigma(\{1\})\geq \mu(A)^2$ (as a consequence of the $L^2$-mean ergodic theorem and the spectral theorem) such that for all $u, v \in \mathcal{O}_K^{\times}$ we have
\[
\int_{\mathcal{M}} f(u) \cdot \overline{f(v)} \ d\sigma(f) = \mu(T_u^{-1}A \cap T_v^{-1}A). 
\]
In particular, provided that none of the linear forms vanish on $m,n$ (which is a set of additive density $0$ in $(m,n)$), we can write
\[
\mu(T_{\ell (m+\alpha n)(m+\beta n)}^{-1}A \cap T_{\ell' mn}^{-1}A) = \int_{\mathcal{M}} f(\ell (m+\alpha n)(m+\beta n)) \cdot \overline{f(\ell' mn)} \ d\sigma(f).
\]
Thus, writing $$S_q:=\left\{(m,n) : mn(m+\alpha n)(m+\beta n) \neq 0 \text{ and } \ell(m+\alpha n)(m+\beta n) \neq \ell'mn\right\}$$ and using the parametrization  discussed above for $x,y$, 
Theorem~\ref{firstreductionplus} follows from the following result.
\begin{theorem}\label{secondreduction}
Let $\sigma$ be a finite positive Borel measure on $\mathcal{M}$ such that 
\begin{equation}\label{positivitymeasures}
    \text{$\sigma(\{1\})>0$ and } \int_{\mathcal{M}} f(u)\cdot \overline{f(v)} \ d\sigma(f)\geq 0 \quad \text{for every } u, v \in \mathcal{O}_K^{\times}.
\end{equation}
Then, for every $\ell, \ell' \in \mathcal{O}_K^{\times}$ and $\alpha, \beta \in \OK$ we have
\begin{equation}\label{multav1}
    \lim_{N \to \infty} \E_{\NN(n),\NN(m) \leq N, \ (m,n) \in S_q}\int_{\mathcal{M}}  f(\ell (m+\alpha n)(m+\beta n)) \cdot \overline{f(\ell' mn)}\ d\sigma(f)>0.
\end{equation}
\end{theorem}
\begin{remark}
The limits in \eqref{multav1} exist thanks to \cite[Theorem~1.12]{S23}, which differs from the situation in the $\N$ case, because of the irreducibility of one of the quadratic forms in $m$ and $n$.
\end{remark}

\subsection{A further break down for Theorem \ref{secondreduction}}\label{s23}

The next part of our proof strategy differs from the path taken in \cite{S23}. We do not attempt to use a decomposition result that works for all elements of $\mathcal{M}$ simultaneously (which, as we discussed, would not cover the case of Pythagorean pairs), but instead divide $\mathcal{M}$ into its aperiodic and pretentious parts. In order to ensure that the splitting is a truly disjoint union, we use the analog of Hal\'asz's theorem for number fields we mentioned: Theorem~\ref{C22intro}. 

We can then use linear concentration estimates that we develop in Section~\ref{SECconcest} to deal with the pretentious part, and results from \cite{S23} to deal with the aperiodic part (which will vanish).

In order to prove Theorem~\ref{secondreduction}, we will take the averages over the grid
\[
\{ (Qm+1, Qn) : m, n \in \mathcal{O}_K^{\times}\},
\]
for some $Q \in \OK$ that will be appropriately chosen later, depending only on the measure $\sigma$. Since we are only concerned with positivity, going along this grid is enough to establish \eqref{multav1}. Now, let $\ell,\ell' \in \mathcal{O}_K^{\times}$ and $\alpha,\beta \in \OK$ be fixed. For $\delta>0$, $f \in \mathcal{M}$ and $Q, m, n \in \mathcal{O}_K^{\times}$, we put
\begin{equation}
A_{\delta}(f,Q; m,n):= w_{\delta}(m,n)\cdot  f\Bigl(\frac{\ell(Qm+1+Q\alpha n)(Qm+1+Q\beta n)}{\ell' (Qm+1)Qn}\Bigr),
\end{equation}
where $w_{\delta}: (\mathcal{O}_K^{\times})^2 \to [0,1]$ is the weight defined in Lemma~\ref{weights} below which is supported on $S_{q}$. Since $0 \leq w_{\delta}\leq 1 $, and we also have the positivity property \eqref{positivitymeasures}, Theorem~\ref{secondreduction} will follow from the following.

\begin{theorem}\label{thirdreduction}
Let $\sigma$ be a Borel probability measure on $\mathcal{M}$ such that $\sigma(\{1\})>0$. Then, there exist $\delta>0$ and $Q \in \mathcal{O}_K^{\times}$ (depending only on $\sigma$) such that
\begin{equation}\label{adeltapos}
\lim_{N \to \infty} \E_{\NN(n),\NN(m) \leq N} \int_{\mathcal{M}} A_{\delta}(f,Q; m,n) \ d\sigma(f)>0.
\end{equation}
\end{theorem}
In order to analyse the limit in \eqref{adeltapos}, we begin by looking into the case where $f$ is aperiodic (the precise definition will be given later in Section~\ref{SEC: numthybackg}). 

\begin{proposition}\label{aperiodiclimit0}
Let $f: \mathcal{O}_K^{\times} \to \U$ be an aperiodic completely multiplicative function. Then, for every $\delta>0$ and $Q \in \mathcal{O}_K^{\times}$ we have
\begin{equation}
\lim_{N \to \infty} \E_{\NN(n),\NN(m) \leq N} A_{\delta}(f,Q; m,n)=0.
\end{equation}
\end{proposition}
Proposition \ref{aperiodiclimit0} is an analog of \cite[Proposition~4.1]{FKM1}.
The proof of Proposition \ref{aperiodiclimit0} is based on Proposition~\ref{P: appendix}, a variation of which was essentially proved in \cite{S23}. We postpone the details until Section \ref{sec: proofmain}.

We now turn our attention to the complement of the aperiodic completely multiplicative functions. As we will deduce from Theorem~\ref{C22intro}, these are exactly the pretentious completely multiplicative functions, so we introduce the following notation:
\begin{equation}
\mathcal{M}_p:=\{ f: \mathcal{O}_K^{\times} \to \S^1 : f \text{ is a pretentious completely multiplicative function}\}.
\end{equation}
Lemma~\ref{measurability} establishes that $\mathcal{M}_p$ is a Borel measurable subset of $\mathcal{M}$. Thus, by Proposition~\ref{aperiodiclimit0} and the dominated convergence theorem, we see that Theorem~\ref{thirdreduction} will follow if we find $\delta>0$ and $Q \in \mathcal{O}_K^{\times}$ such that the analog of \eqref{adeltapos} holds, replacing $\mathcal{M}$ with $\mathcal{M}_p$, i.e.,
\begin{equation}\label{adeltaposPret}
\lim_{N \to \infty} \E_{\NN(n),\NN(m) \leq N} \int_{\mathcal{M}_p} A_{\delta}(f,Q;m,n) \ d\sigma(f)>0.
\end{equation}

When $f$ is pretentious, it exhibits periodicity, which can be exploited with a suitable choice of $Q$. This simplifies matters considerably for the integrand that appears in \eqref{adeltaposPret}. In order to take advantage of this periodicity, we develop, in Section~\ref{SECconcest} a concentration estimate in Proposition~\ref{p25}. We do not restate it here, as it is fairly long and it requires introducing notation and technical terms that we will see later.
We mention in passing that we will use it mostly for the value $a=1$, and that it will be relevant for us that the implicit constant in the statement is independent of $K$; as well as having independence of the function $F_N(f,k)$ from $Q$ (all these terms will be introduced later).

To establish \eqref{adeltaposPret}, we further split the integral into two parts. On the one hand, we consider the multiplicative functions that are not purely Archimedean characters $(\NN(u)^{i\tau})_{u \in \mathcal{O}_K^{\times}}$, $\tau \in \R$, where the concentration estimate in Proposition~\ref{p25} allows us to show that their contribution is essentially non-negative if $Q$ is highly divisible. The other part is supported on Archimedean characters $\mathcal{A}$, which we define in \eqref{archimedean} below. Using the fact that $\sigma(\{1\})>0$ and for some $\delta>0$ small enough, the weight $w_{\delta}$ nullifies the effect of non-trivial Archimedean characters. 

To carry this out we will make good use of the properties of the multiplicative F\o lner sequence defined in \eqref{folnerMult} below. On a first reading, one can think of it as a suitable analog of the multiplicative F\o lner sequence
\[
\Psi_K:=\{ p_1^{a_1}\dots p_k^{a_k} : K \leq a_i \leq 2K, i=1,\dots,K\}
\]
in the integers.
Let
\begin{equation}\label{archimedean}
\mathcal{A}:= \left\{ (\NN(u)^{it})_{u \in \mathcal{O}_K^{\times}} : t \in \R\right\}.
\end{equation}
We will later show (it will follow from the proof of Proposition~\ref{fffss} given in Section~\ref{sec: proofmain}) that $\mathcal{A}$ is a Borel measurable subset of $\mathcal{M}$, but assuming that this is the case for now, the next step is to use the concentration estimate in Proposition~\ref{p25} to obtain the following result.

\begin{proposition}\label{zeroaveragesadeltabdelta}
Let $f \in \mathcal{M}_p \setminus \mathcal{A}$, $\delta>0$, $\ell,\ell' \in \mathcal{O}_K^{\times}$, and $\Phi_M$ as in \eqref{folnerMult}. Then,
\[
\lim_{M \to \infty} \E_{Q \in \Phi_M} \lim_{N \to \infty} \E_{\NN(n),\NN(m) \leq N} A_{\delta}(f,Q; m,n)=0.
\]
\end{proposition}
Proposition~\ref{zeroaveragesadeltabdelta} can be viewed as a variation of \cite[Lemma~2.6]{FKM1} 
and its detail is postponed to Section \ref{sec: proofmain}. The essential idea is to use the concentration estimate in Proposition~\ref{p25} to replace it by an expression of the form $C_{\ell,\ell'}\cdot g(Q)\cdot \NN(Q)^{i\tau}$, for some $C_{\ell,\ell'} \in \U$ and $\tau \in \R$. Moreover, we can have $g \notin \mathcal{A}$, so the remaining outer limit will give us convergence to $0$ as $M \to \infty$ using Lemma~\ref{multav0}, whose proof is also deferred to Section~\ref{sec: proofmain}.

With all this taken into consideration we can deduce, using the dominated convergence theorem twice (as the two relevant limits exist), we have the following.

\begin{corollary}
Let $(\Phi_M)$ and $\mathcal{A}$ be given by \eqref{folnerMult} and \eqref{archimedean} respectively. Let $\sigma$ be a Borel probability measure on $\mathcal{M}_p$. Then, for every $\delta>0$ we have
\[
\lim_{M \to \infty} \E_{Q \in \Phi_M} \lim_{N \to \infty} \E_{\NN(n),\NN(m) \leq N} \int_{\mathcal{M}_p \setminus \mathcal{A}} A_{\delta}(f,Q; m,n) \ d\sigma(f)=0. 
\]
\end{corollary}
Finally, only the functions from $\mathcal{A}$ remain. Here is where the weight $w_{\delta}$ helps with the positivity.

 \begin{proposition}\label{fffss}
Let $\sigma$ be a Borel probability measure on $\mathcal{M}$ such that $\sigma(\{1\})>0$ and $\mathcal{A}$ be as in \eqref{archimedean}. Then, there exist $\delta$ and $\rho$, depending only on $\sigma$, such that
\begin{equation}
\liminf_{N \to \infty} \inf_{Q \in \mathcal{O}_K^{\times}} \text{Re} \left( \E_{\NN(n),\NN(m) \leq N} \int_{\mathcal{A}} A_{\delta}(f,Q;m,n)\right) \geq \rho.
\end{equation}
\end{proposition}

Proposition \ref{fffss} can be viewed as a variation of 
\cite[Lemma 2.8]{FKM1} and
again we defer its proof  to Section \ref{sec: proofmain}. It is here where we make essential use of the weight function $w_{\delta}$ to obtain positivity for the averages whose limit does not necessarily exist because of the nature of the archimedean characters.

Assuming the previous results, Theorem~\ref{thirdreduction} easily follows from the fact that we can find $\delta_0,\rho_0>0$ so that
\[
\liminf_{M \to \infty}\E_{Q \in \Phi_M} \lim_{N \to \infty} \E_{\NN(n),\NN(m) \leq N} \int_{\mathcal{M}_p \setminus \mathcal{A}} A_{\delta_0}(f,Q; m,n) \ d\sigma(f) \geq \rho_0.
\]

\subsection{Averaging schemes}\label{s:as}
In this subsection we wish to highlight the fact that the choice of the averaging scheme is very important to us for a number of reasons. 
In order to make the discussion easier to follow, we will focus our attention on the Gaussian integers, but the same points we shall discuss equally apply to other quadratic imaginary fields. 
To study the averages of a multiplicative function of the form
$$\lim_{N\to\infty}\frac{1}{\vert\Phi_{N}\vert}\sum_{u\in\Phi_{N}}f(u),$$
there are at least 4 natural choices for the sequence $(\Phi_{N})_{N}$:
\begin{enumerate}
    \item   symmetric  balls $D_{N}:=\{a+bi\colon a,b\in\Z, a^{2}+b^{2}\leq N\}$;
    \item   symmetric boxes $B_{N}:=\{a+bi \colon -N\leq a,b\leq N\}$;
     \item  asymmetric  balls $D_{N}^{+}:=\{a+bi\colon a,b\in\N, a^{2}+b^{2}\leq N\}$;
    \item   asymmetric boxes $B_{N}^{+}:=\{a+bi \colon 1\leq a,b\leq N\}$.
\end{enumerate}

In \cite{S18,S23}, all the averages are taken over boxes. The  advantage  in doing so  is that the box average is well defined for any number field while the ball averages can only be defined for quadratic imaginary fields, as for a general number field $K$, the set $\{u\in\OK\colon \NN(u)\leq N\}$ can be infinite. However,  we are unable to prove an analog of Theorem \ref{characterization} if we define aperiodic functions using box averages as was done in \cite{S18,S23}, mainly because we do not know if an analog of Hal\'asz's Theorem (Theorem \ref{C22intro}) holds for box averages. Therefore, in this paper we favor ball averages over box averages.

Next we explain why it is more convenient to work with symmetric balls instead of the asymmetric ones. Consider the multiplicative function $f(u)=\frac{u}{\NN(u)^{1/2}}$. One can show that $f$ does not have finite distance to the product of Dirichlet characters and Archimedean characters either.  On the other hand, 
the average of $f$ over asymmetric boxes does not converge to 0 as $f$ takes values only in the first quadrant. So an analog of Hal\'asz's theorem fails to hold in this setting, unless we expand the definition of pretentious functions to include functions $f$ of this form (however we do not have this issue for symmetric ball averages as the symmetric ball average of $f$ does converge to 0, as can be checked with a Riemann sums argument).
Because of the above reasons, in this paper we choose to work with (i): averages over symmetric balls.

\section{Number theoretical background}\label{SEC: numthybackg}
In this section we will review some basic number theoretical notions and notation, and also introduce some basic estimates we will need to make use of in the sequel.

An \emph{(algebraic) number field} $K$ is a finite degree (and hence algebraic) field extension of the field of rational numbers $\Q$. The \emph{ring of integers} $\OK$ of a number field $K$ is the ring of all \emph{integral elements} in $K$ (i.e., roots of polynomials with integer coefficients and leading coefficient 1). Let $D=[K:\Q]$ denote the degree of the extension. It is classical that there exists an \emph{integral basis} $\mathcal{B}=\{b_{1},\dots,b_{D}\}$ of $\OK$, i.e., a basis of the $\Q$-vector space $K$ such that each element $x\in\OK$ can be uniquely represented as $x=\sum_{i=1}^{D}c_{i}b_{i}$ for some $c_{i}\in\Z$.   

Let $\iota\colon\Z^{D}\to\OK$ be the map given by $\iota(n_{1},\dots,n_{D})=n_{1}b_{1}+\dots+n_{D}b_{D}$. For $x\in K$, let $A_{x}$ be the unique $d\times d$ matrix such that $\begin{bmatrix}
    xb_{1} \\ \vdots\\ xb_{D}
\end{bmatrix}=A_{x}\begin{bmatrix}
    b_{1} \\ \vdots\\ b_{D}
\end{bmatrix}$. The $K$-\emph{norm} of $x\in K$ is defined to be $\NN_{K}(x):=\det(A_{x})$. When there is no risk of confusion regarding the underlying field $K$, we simply write $\NN$ instead of $\NN_{K}$.

In this paper, our main focus is on quadratic fields, i.e., number fields of the form $K=\Q(\sqrt{-d})$ for some squarefree $d\in\Z$. 
In this case, we have $\OK=\Z[\tau_{d}]$ and every $z\in\OK$ can be written as $m+n\tau_{d}$ for some $m,n\in\Z$ in a unique way. The element $\tau_d$ is given by $\tau_{d}=\sqrt{-d}$ if $d\equiv 1,2 \mod 4$ and $\tau_{d}=\frac{1+\sqrt{-d}}{2}$ if $d\equiv 3 \mod 4$. Moreover, we have that $\NN(m+n\tau_{d})$ is equal to $m^{2}+dn^{2}$ if $d\equiv 1,2 \mod 4$ and equal to $m^{2}+mn+\frac{d+1}{4}n^{2}$ if $d\equiv 3 \mod 4$.

We will use the following classical result on properties of the norm $\NN$.

\begin{lemma}
    Let $K$ be a number field. Then $\NN(xy)=\NN(x)\NN(y)$ for all $x,y\in K$. Also, for any $x \in \OK$ we have $\NN(x)\in\Z$.
\end{lemma}

We say that $\e\in \OK$ is a unit if $\NN(\e)=\pm 1$. 
It follows from Dirichlet's unit theorem \cite[Theorem~8.1]{neukirch99} that $\OK$ has finitely many units if and only if $K=\Q$ or $\Q(\sqrt{-d})$ for some square-free $d\in\N$. Moreover, 
we have a complete description for the units in this case.

\begin{lemma}\label{l32}
    Let $K=\Q(\sqrt{-d})$ for some squarefree $d\in\Z$. 
    \begin{enumerate}
        \item If $d=1$, then the units of $\OK$ are $\pm 1, \pm i$.
        \item If $d=3$, then the units of $\OK$ are $e^{k\pi i/3}, 0\leq k\leq 5$.
        \item If $d>0, d\neq 1,3$, then the units of $\OK$ are $\pm 1$.
    \end{enumerate}
\end{lemma}
This follows easily from computing the elements in $\OK$ whose norm is equal to $\pm 1$ and solving the resulting diophantine equations, which is straightforward, as in the imaginary quadratic case the norm is non-negative. 

Let $K$ be a number field. We use $\mathfrak{I}(K)$ to denote the set of all  ideals of $\OK$, $\mathfrak{Q}(K)$ to denote the set of all fractional ideals of $K$ (recall that a fractional ideal has the form $x^{-1}I$ for some $x \in \mathcal{O}_K^{\times}$ and non-trivial ideal $I \subseteq \OK)$. 
For $I\in \mathfrak{I}(K)$, we say that its \emph{norm} is $\NN(I):=[\OK:I]$. In the case where $x^{-1}I\in \mathfrak{I}(K)$, the \emph{norm} is given by $\NN(x^{-1}I):=\NN(I)/\NN(x)$.

We say that two ideals $I, J \in \mf I(K)$ are \emph{coprime} if $I+J=(1)$. For Dedekind domains, this is equivalent to $I,J$ sharing no elements in their factorization into prime ideals.
 Let $a,b\in\OK^{\times}$, $\mf p\in \mf I(K)$ and $k\in\N$. We write $\mf p\vert a$ if $a\in\mf p$ (or equivalently $\mf p\vert (a)$). Write $\mf p^{k}|| a$ if $k$ is the largest integer for which $\mf p^{k}\vert a$.

We recall a definition.
\begin{definition}
Let $K$ be a number field. 
We say that the quotient $\mf Q(K)/\mf I(K)$ is the \emph{ideal class group of} $K$.
\end{definition}
Given a number field $K$, its ideal class group $G$ is always finite (e.g., see \cite[Theorem~6.3]{neukirch99}), and we say that the order of $G$ is the \emph{class number} of $K$.

The ideal class group of a number field $K$ depends on the algebraic properties of the number field under consideration. For example, if $K=\Q(i)$, its ideal class group is trivial (this is a consequence of $\Z[i]$ being a PID), but for $\Q(\sqrt{-5})$ it is isomorphic to $\Z/2\Z$.

The following lemma will be used in the sequel.
\begin{lemma}\label{cl2}
    Let $K$ be a number field and $D=[K:\Q]$. For any $I\in\mathfrak{I}(K)$, there exists a finite union of $D$-dimensional infinite arithmetic progressions $P$ such that for $u,v\in\OK^{\times}$
$$\frac{(u)}{(v)}I\in\mathfrak{I}(K)\Leftrightarrow \frac{(u)}{(v)}=\frac{(u')}{(\NN(I))} \text{ for some } u'\in P.$$
In fact, $P$ is the set of $u'$ such that $\NN(I)^{D-1}\vert\NN(u')$.
\end{lemma} 
\begin{proof}
By multiplying with conjugates of $v$ if necessary, we may assume without loss of generality that $v\in\Z$. Then $\frac{(u)}{(v)}I\in\mathfrak{I}(K)\Leftrightarrow v^{D}\vert\NN(u)\NN(I)$.
Write $\NN(I)=u_{1}^{t_{1}}\dots u_{n}^{t_{n}}$ for some disticnt prime numbers $u_{i}$ and some powers $t_{i}\in\N$. 

We now proceed by cases. If a prime number $t$ is such that $t\vert v$ but $t\nmid \NN(I)$, then $t^{D}\vert \NN(u)$ and so $u/t\in\mathcal{O}_{K}$. So $\frac{(u)}{(v)}=\frac{(u/t)}{(v/t)}$. If $u_{i}^{t_{i}+1}\vert q$ for some $i$, then $p^{D}\vert\NN(u)\NN(I)$ implies that $u_{i}^{D}\vert \NN(p)$ and again $\frac{(u)}{(v)}=\frac{(u/u_{i})}{(v/u_{i})}$.
Thus, we may freely assume that $v=\NN(I)$.

We have that $v^{D}\vert\NN(u)\NN(I)\Leftrightarrow \NN(I)^{D-1}\vert\NN(u)$. The set of such $u$ is clearly a finite union of $D$-dimensional infinite arithmetic progressions. This completes the proof. 
\end{proof}

We conclude this section with a counting property on the number of 
 ideals in a given ideal class. These results are given in the following lemmas.

 \begin{lemma}[Theorem~2, \cite{MVO07}]\label{mv7}
     Let $K$ be a number field of degree $D$ and $C$ be an ideal class of it. For $x>0$, let $N(x,C)$ denote the number of ideals in the class $C$ whose norms are at most $x$. Then there exist constants $C_1, C_2>0$ depending on the number field $K$ only, such that 
     $$C_{1}(x^{\frac{1}{D}}-C_{2})^{D}\leq N(x,C)\leq C_{1}(x^{\frac{1}{D}}+C_{2})^{D}$$
      for all $x>1$ for some $C_{1}>0$ and $0<C_{2}<1$ depending only on $K$.
 \end{lemma}

 As a consequence of Lemma \ref{mv7}, we have (see  also \cite[Chapter 6, Theorem 39]{marcusbook} for a reference):
\begin{corollary}\label{lem: bounds for number of ideals of norm at most x}
Let $K$ be a number field of degree $D$. Then, there exists a universal constant $\gamma_K>0$ and some $0 \leq \eta<1$ 
such that for any ideal class $C$ of $K$, we have that 
\[
\frac{1}{x}\sum_{\mf u\in C\colon \NN(\mf u) \leq x}1 = \gamma_K+o(x^{\eta}).
\]
\end{corollary}
\begin{remark}
In particular, with the terminology from \cite{lucht2001}, Corollary~\ref{lem: bounds for number of ideals of norm at most x} above implies that $\mathfrak I(K)$ is an Axiom A arithmetic group using the standard norm of an ideal for $A=\gamma_{K}$. 
\end{remark}

 \begin{lemma}\label{mv8}
     For any number field $K$, there exists a constant $C$ such that for any  $x>0$, the number of ideals of $\OK$ which is a power of a prime ideal and whose norm is $x$ is at most $[K\colon\Q]$.
 \end{lemma}
 \begin{proof}
 Let $D:=[K\colon\Q]$.
     By \cite[Lemma~2.9]{S23}, the norm of every prime ideal is a power of prime in $\N$. So we may assume that $x=p^{r}$ for some $r\in\N$ and prime $p\in\N$. Suppose that $\NN(\mf p^{k})=p^{r}$, then we must have that $r=kt$ for some $t\in\N$ and $\NN(\mf p)=p^{t}$. Again by \cite[Lemma~2.9]{S23}, we must have that $t\leq D$ and there are at most $D$ such $\mf p$. Now for each such $\mf p$, there is at most one choice of $k\in\N$ such that $\NN(\mf p^{k})=p^{r}$. So the number of ideals of the form $\mf p^{k}$ with norm equal to $p^{r}$ is at most $D$.
 \end{proof}


\begin{convention}
Throughout the paper, $\mf p$ and $\mf q$ always denote prime ideals of $\mathcal{O}_K$. Whenever we sum over $\mf p$ or $\mf q$, this sum is assumed to be taken along prime ideals.
\end{convention}

We conclude this section by some estimates related to the distributions of prime ideals which will be used in later sections.

\begin{lemma}\label{l2}
As $N\to\infty$, we have
\begin{enumerate}
    \item $\sum_{\NN(\mf p^{k})\leq N} 1=O\left(\frac{N}{\log N}\right);$
    \item $\sum_{\NN(\mf p^{k})\leq N} \NN(\mf p^{k})^{-1}=O(\log\log N);$
 \item $\sum_{\NN(\mf p^{k}\mf q^{\ell})\leq N, \mf p\neq \mf q} 1=O\left(\frac{N\log\log N}{\log N}\right)$.
\end{enumerate}
%
%
\end{lemma}
\begin{proof}
For $i\in\N$, let  $\Pi_{i}(N)$ denote the set of integers  $x\in\{1,\dots,n\}$ with at most $i$ distinct factors. Let $\pi_{i}(N):=\vert\Pi_{i}(N)\vert$. 
Throughout the proof, we make use of the following facts without explicitly stating them:
\begin{itemize}
    \item It is known that (see \cite[Chapter II.6]{tenebaum_2015}, for example)
$$\pi_{i}(N)\sim C_{i}\frac{N(\log\log N)^{i-1}}{\log N};$$
\item By the Mertens' second theorem (see for example \cite[Chapter I.1, Theorem~9]{tenebaum_2015}), 
$$\sum_{p\leq N, \text{ $p$ is a prime}} p^{-1}=\log\log N+O(1);$$
\item By \cite[Lemma~2.9]{S23}, the norm of every prime ideal is of the form $p^{r}$ for some prime $p\in\N$ and some $1\leq i\leq [K:\Q]$;
\item By Lemma \ref{mv8}, we have 
$\sum_{\NN(\mf p^{k})=a} 1\ll 1$ for all $a\in\N$;
\item For any $I\subset\N\backslash\{1\}$ and $s\geq 1$, we have $$\sum_{a\in I, \text{ $a$ is a power of prime}}a^{-s}\leq \sum_{a\in I, \text{ $a$ is a prime}}a^{-s}+O(1),$$
since $\sum_{k=2}^{\infty}\sum_{n=2}^{\infty}n^{-sk}\leq \sum_{k=2}^{\infty}\sum_{n=2}^{\infty}n^{-k}=\sum_{n=2}^{\infty}\frac{1}{n(n-1)}=1.$
\end{itemize}
Let us show each of the asserted statements. 
First, for part (i), we have
$$\sum_{\NN(\mf p^{k})\leq N} 1=\sum_{a\in \Pi_{1}(N)}\sum_{\NN(\mf p^{k})=a} 1\ll \sum_{a\in \Pi_{1}(N)} 1=\pi_{1}(N)=O\left(\frac{N}{\log N}\right).$$
Next, part (ii) follows from the fact that
\begin{eqnarray*}
    \sum_{\NN(\mf p^{k})\leq N} \NN(\mf p^{k})^{-1} & = & \sum_{a\in \Pi_{1}(N)}\sum_{\NN(\mf p^{k})=a} a^{-1}
      \\&\ll & \sum_{a\in \Pi_{1}(N)} a^{-1}\ll \sum_{p\leq N, p \text{ is a prime}} p^{-1}=O(\log\log N).
\end{eqnarray*}
Lastly, for part (iii), let $\Pi_{i,T}(N)$ denote the set of natural numbers no larger than $N$ which are of the form $p^{k}q^{\ell}$ for some primes $p,q$ and some $0\leq k, \ell \leq T$. Then,
\begin{equation}\label{mmvv6}
    \begin{split}
      &\qquad \sum_{\NN(\mf p^{k}\mf q^{\ell})\leq N, \mf p\neq \mf q} 1=\sum_{a,b\in \Pi_{1,D}(N),k,\ell\in\N,a^{k}b^{\ell}\leq N}\sum_{\NN(\mf p)=a, \NN(\mf q)=b}1
       \ll \sum_{a,b\in \Pi_{1,D}(N),k,\ell\in\N,a^{k}b^{\ell}\leq N} 1,
    \end{split}
\end{equation}
where $D:=[K:\Q]$.
Fix any $a,b\in \Pi_{1, D}(N),k,\ell\in\N$ with $a^{k}b^{\ell}\leq N$. If $a,b$ are powers of different primes, then $a^{k}b^{\ell}$ belongs to $\pi_{2}(N)$, and for each $x\in \pi_{2}(N)$, $x$ can be written in the form $a^{k}b^{\ell}$ for some  $a,b\in \Pi_{1,D}(N),k,\ell\in\N$ in at most $2D^{2}$ ways. If $a,b$ are powers of the same prime, then $a^{k}b^{\ell}$ belongs to $\pi_{1}(N)$. In this case, for each $x=p^{r}\in \pi_{1}(N)$ with $p$ being a prime, the number of ways $x$ can be written in the form $a^{k}b^{\ell}$ for some $a,b\in \Pi_{1,D}(N),k,\ell\in\N$ is  at most $D^{2}$. So  the right hand side of (\ref{mmvv6}) can be bounded by
 $$2D^{2}\pi_{2}(N)+D^{2}\sum_{p^{r}\in \pi_{1}(N)} 1
       \leq O\left(\frac{N\log\log N}{\log N}\right)+\sum_{r=1}^{\lceil\frac{\log N}{\log 2}\rceil}\pi_{1}({N}^{1/r}).$$
The estimate
\begin{eqnarray*}
    \sum_{r=1}^{\lceil\frac{\log N}{\log 2}\rceil}\pi_{1}({N}^{1/r})
      & \ll & \sum_{r=1}^{\lceil\frac{\log N}{\log 2}\rceil}\frac{r{N}^{1/r}}{\log N}
      \leq \frac{N}{\log N}+\sum_{r=2}^{\lceil\frac{\log N}{\log 2}\rceil}\frac{r\sqrt{N}}{\log N}
      \\& \leq & \frac{N}{\log N}+\log N\sqrt{N}
      \leq O\left(\frac{N\log\log N}{\log N}\right),
\end{eqnarray*} 
gives the conclusion.
\end{proof}

\section{A classification result for completely multiplicative functions}\label{ss5}

The purpose of this section is to obtain a useful classification result for multiplicative functions, which will be the cornerstone of our analysis of the set $\mathcal{M}$ which we must split into different pieces according to the asymptotic behavior of the multiplicative function $f$. We begin with extensions of multiplicative functions.

\subsection{Extensions of multiplicative functions}
In this section, we prove Theorem \ref{characterization}.
Before that, we need to classify all the extensions of a completely multiplicative function $f$. 
Let $D=\mathcal{O}_K^{\times}$ or $\mathfrak{I}(K)$.
We say that $f\colon D\to\U$ is \emph{completely multiplicative} (on $D$) if $f(mn)=f(m)f(n)$ for all $m,n\in D$. 
%
%
For any completely multiplicative function $f\colon \mathcal{O}_K^{\times} \to\U$
with $f(\e)=1$ for all units $\e$,\footnote{It is clear that if $f$ admits an extension, then   $f(\e)=\tilde{f}((\e))=\tilde{f}((1))=f(1)=1$ for all units $\e$, so naturally we only define extensions of multiplicative functions $f$ if they satisfy this additional condition.} we say that a completely multiplicative function $\tilde{f}\colon \mathfrak{Q}(K)\to\U$ is an \emph{extension} of $f$ if  $\tilde{f}((x))=f(x)$ for all $x\in \mathcal{O}_K^{\times}$. One of the key ideas in this paper is the study of averages of completely multiplicative functions on a non-uniquely factorizable domain which is done by passing to extended versions of the original multiplicative functions.

Let $G$ be the ideal class group of $K$. By \cite[Theorem~6.3]{neukirch99}, $G$ is a finite group, and it is also clearly abelian by construction. Therefore, by the classification of finite abelian groups, we can find $k\in\N$, 
    and generators $g_{1},\dots,g_{k}$ of $G$ of orders $d_{1},\dots,d_{k}\in\N$ such that $G=\langle g_1,\dots,g_k\rangle$.
   
    For each $1\leq i\leq k$, let $I_{i}\in \mathfrak{Q}(K)$ be a representative of $g_{i}$. We may assume without loss of generality that $I_{i}\in\mathfrak{I}(K)$. If so, then we say that $(I_{1},\dots,I_{k})$ is an \emph{ideal class group representation} for $\OK$. 
    It follows that $I_{i}^{d_{i}}=(x_{i})$ for some $x_{i}\in \mathcal{O}_{K}^{\times}$.
It is clear that if $\tilde{f}$ is an extension of $f$, then we must have that $\tilde{f}(I_{i})^{d_{i}}=f(x_{i})$ for all $1\leq i\leq k$. 
Conversely, we have the following.

\begin{proposition}\label{prop: extensions}
    Let $f\colon \mathcal{O}_K^{\times}\to\S$ be a completely multiplicative function  with $f(\e)=1$ for all units $\e$. Suppose that 
    $(I_{1},\dots,I_{k})$ is an ideal class representation of $\OK$ 
 with
 $I_{i}^{d_{i}}=(x_{i})$ for some  $x_{i}\in \mathcal{O}_{K}^{\times}$ for  $1\leq i\leq k$. Then for any $a_{1},\dots,a_{k}\in\C$ with $a_{i}^{d_{i}}=f(x_{i}), 1\leq i\leq k$, there exists a unique extension $\tilde{f}$ of $f$ such that $\tilde{f}(I_{i})=a_{i}$ for all $1\leq i\leq k$. 
\end{proposition}
\begin{proof} 
The uniqueness part is obvious and so we now prove the existence of such an extension. Note that every ideal can be written as $\frac{(x)}{(y)}I$, where $I=I_{1}^{b_{1}}\dots I_{k}^{b_{k}}$, and for some $x,y\in \mathcal{O}_K^{\times}$ and $0\leq b_{i}\leq d_{i}-1$. 
By   Lemma \ref{cl2}, we may assume that $y=\NN(I)$. For such an ideal,
define
    $$\tilde{f}\left(\frac{(x)}{(\NN(I))}I\right):=f(x)f(\NN(I))^{-1}a_{1}^{b_{1}}\dots a_{k}^{b_{k}}$$
for all $x,y\in \mathcal{O}_K^{\times}$ and $0\leq b_{i}\leq d_{i}-1$, which is well defined since $x$ is uniquely determined up to a unit and $f(\e)=1$ for all units $\e$. 
Since clearly $a_{1},\dots,a_{k}\in\S$, it follows that $\tilde{f}$ takes values in $\S$.

It is clear that $\tilde{f}((x))=f(x)$ for all $x\in \mathcal{O}_K^{\times}$ (in which case we must have that $b_{1}=\dots=b_{k}=0$). 
We now show that $\tilde{f}$ is multiplicative.

Let $x,x'\in \mathcal{O}_K^{\times}$ and $0\leq b_{i},b'_{i}\leq d_{i}-1$ be such that $J:=\frac{(x)}{(\NN(I))}I$ and $J':=\frac{(x')}{(\NN(I'))}I'$ are ideals, where $I=I_{1}^{b_{1}}\dots I_{k}^{b_{k}}$ and $I'=I_{1}^{b'_{1}}\dots I_{k}^{b'_{k}}$.
Put $c_{i}=b_{i}+b'_{i}$ if $c_{i}\leq d_{i}-1$ and $c_{i}=b_{i}+b'_{i}-d_i$ otherwise. Denote $r_{i}=b_{i}+b'_{i}-c_{i}$.
Then 
$$JJ'=\frac{(xx'x_{1}^{r_{1}}\dots x_{k}^{r_{k}})}{(\NN(I_{1}^{b_{1}+b'_{1}}\dots I_{k}^{b_{k}+b'_{k}}))}I_{1}^{c_{1}}\dots I_{k}^{c_{k}}.$$
Since $JJ'$ is an ideal, by  Lemma \ref{cl2}, we have that 
$$JJ'=\frac{(y)}{(\NN(I_{1}^{c_{1}}\dots I_{k}^{c_{k}}))}I_{1}^{c_{1}}\dots I_{k}^{c_{k}}$$
for some $y\in\mathcal{O}_K^{\times}$.
Therefore, we have that 
$$y=xx'x_{1}^{r_{1}}\dots x_{k}^{r_{k}}/\NN(I_{1}^{r_{1}})\dots \NN(I_{k}^{r_{k}}).$$
By definition,
\begin{equation*}
    \begin{split}
        &\qquad \tilde{f}(JJ')
        =f(y)f(\NN(I_{1}^{c_{1}}\dots I_{k}^{c_{k}}))^{-1}a_{1}^{c_{1}}\dots a_{k}^{c_{k}}.
    \end{split}
\end{equation*}
On the other hand,
$$\tilde{f}(J)\tilde{f}(J')=f(x)f(\NN(I_{1}^{b_{1}}\dots I_{k}^{b_{k}}))^{-1}f(x')f(I_{1}^{b'_{1}}\dots I_{k}^{b'_{k}})a_{1}^{b_{1}+b'_{1}}\dots a_{k}^{b_{k}+b'_{k}}.$$
One can easily check that $\tilde{f}(JJ')=\tilde{f}(J)\tilde{f}(J')$, completing the proof.
%
\end{proof}

\subsection{Characterizations for aperiodic completely multiplicative functions}
 
We now begin to prove Theorem \ref{characterization} by providing a list of equivalent definitions of aperiodic multiplicative functions.
We begin with a short lemma that justifies an assumption we will make many times throughout the paper: that $f$ can be assumed to be trivial on units.
\begin{lemma}\label{tv}
    Let $d\in\N$ be squarefree and $f\colon \Z[\tau_{d}]^{\times}\to\C$ be a  completely multiplicative function. If $f(\e)\neq 1$ for some unit $\e$. Then $f$ is aperiodic.
 \end{lemma}
 \begin{proof}
 This follows from the F\o lner property that the balls $B_N:= \{ u \in \OK : \NN(u)\leq N\}$ enjoy. Indeed, notice that each $B_N$ has the property that for any unit $\e \in \Z[\tau_{d}]^{\times}$, it is $\e B_N=B_N$, $N \in \N$. 
Let $L_{N}$ be the average in (\ref{eq: aperiodicdef}) without taking the limit over $N$. It is clear that $L_{N}=f(\e) L_{N}$. Thus, if $f(\e)\neq 1$ for some unit $\e$, then $L_{N}=0$ for all $N$. So $f$ is aperiodic. 
 \end{proof}

The following proposition gives a list of equivalent definitions for aperiodic functions.
\begin{proposition}\label{prop: dircharequivalence}
 Let $d\in\N$ be squarefree and $f\colon \Z[\tau_{d}]^{\times}\to\S$ be a completely multiplicative function with  $f(\e)=1$ for all units $\e$. 
    The following are equivalent. 
    \begin{enumerate}
        \item $f$ is aperiodic.
        \item The following limit vanishes.
        $$\lim_{N\to\infty}\sup_{P\in \mathcal{AP}[\tau_{d}]}\left|\E_{u\in\Z[\tau_{d}]^{\times}\colon \NN(u)\leq N}\mathds{1}_{P}(u)f(u)\right|=0.$$
        \item For any Dirichlet character $\chi\colon \Z[\tau_{d}]^{\times}\to\U$, we have that 
        \begin{equation}\nonumber
        \lim_{N\to\infty}\E_{u\in\Z[\tau_{d}]^{\times}\colon \NN(u)\leq N}(f\chi)(u)=0.
        \end{equation}
         \item For any modified Dirichlet character $\chi'\colon \Z[\tau_{d}]^{\times}\to\S$, we have that $$\lim_{N\to\infty}\E_{u\in\Z[\tau_{d}]^{\times}\colon \NN(u)\leq N}(f\chi')(u)=0.$$
        \item For any modified Dirichlet character $\chi'\colon \Z[\tau_{d}]^{\times}\to\S$ 
        and any extension $\widetilde{f\chi'}$ of $f\chi'$,
       \footnote{By Proposition \ref{prop: extensions}, $f\chi'$ admits $\vert G\vert$ extensions.}
        we have that 
        $$\lim_{N\to\infty}\E_{\mathfrak{v}\in \mathfrak{I}(\Q(\sqrt{-d}))\colon \NN(\mathfrak{v})\leq N}\widetilde{f\chi'}(\mathfrak{v})=0.$$
    \end{enumerate}
\end{proposition}

We remark that in (v), one can not replace modified Dirichlet characters by Dirichlet characters, since a $\U$-valued multiplicative function may not admit an extension.

We will split the proof of Proposition~\ref{prop: dircharequivalence} into a series of lemmas.
\begin{lemma}
In Proposition \ref{prop: dircharequivalence}, $(i)$ is equivalent to $(ii)$.
\end{lemma}
\begin{proof}
As was the case in \cite[Appendix A]{S23}, 
the direction (ii)$\Rightarrow$(i) is trivial. We show that (i)$\Rightarrow$(ii).

Suppose that we can find an infinite sequence $(N_k)$ of natural numbers, a sequence of 
$(P_k)$ in $\mathcal{AP}[\tau_{d}]$ and $\varepsilon>0$ such that
\[
\left| \E_{ \NN(u) \leq N_k} \mathds{1}_{P_k}(u) f(u)\right| \geq \varepsilon.
\]


 Assume that $P_{k}=
\{(a_{k,1}m+b_{k,1})+(a_{k,2}n+b_{k,2})\tau_{d}\in \Z[\tau_{d}]\colon m,n\in\Z\}$
        for some $a_{k,1},b_{k,1},a_{k,2},b_{k,2}\in\Z$ with $a_{k,1},a_{k,2}\in\N$. Then $\frac{\vert P_{k}\vert}{\vert \{u\colon \NN(u)\leq N_{k}\}\vert}\geq \varepsilon$. On the other hand, since $m$ and $n$ can take at most $\lceil\frac{10\sqrt{N_{k}}}{a_{k,1}}\rceil$ and $\lceil\frac{10\sqrt{N_{k}}}{a_{k,2}}\rceil$ consecutive integers in $P_{k}$, we have that
        $$\frac{\vert P_{k}\vert}{\vert \{u\colon \NN(u)\leq N_{k}\}\vert}\leq \frac{(\frac{10\sqrt{N}_{k}}{a_{k,1}}+1)(\frac{10\sqrt{N}_{k}}{a_{k,2}}+1)}{C_{d}N_{k}}\leq \frac{(\frac{10}{a_{k,1}}+\frac{1}{\sqrt{N_{k}}})(\frac{10}{a_{k,2}}+\frac{1}{\sqrt{N_{k}}})}{C_{d}}$$
        for some $C_{d}>0$ depending only on $d$. So 
        $(\varepsilon C_{d})^{1/2}\leq \frac{10}{\min\{a_{k,1},a_{k,2}\}}.$
If $k$ is sufficiently large, then $\min\{a_{k,1},a_{k,2}\}\ll_{d, \varepsilon} 1$.
By the pigeonhole principle, we may assume that all the $P_{k}$ are the same by passing to a subsequence if necessary. This contradicts to (i) and we are done.      
\end{proof}

Next we show the equivalence between (i) and (iii).
\begin{lemma}\label{lem: dircharequivalence}
In Proposition \ref{prop: dircharequivalence}, $(i)$ is equivalent to $(iii)$.
\end{lemma}

\begin{proof}


We first show that (iii)$\Rightarrow$(i).
Let $P$ be the set given by (\ref{eq:adcd}). 
It suffices to show that
\[
\lim_{N\to\infty}\E_{\NN(n_1+n_2\tau_{d}) \leq N}\mathds{1}_{a_{1}\Z+b_{1}}(n_1)\mathds{1}_{a_{2}\Z+b_{2}}(n_2)f(n_1+n_2\tau_{d})=0.
\]
By taking common denominators, one can write $\mathds{1}_{a_{1}\Z+b_{1}}(n_1)\mathds{1}_{a_{2}\Z+b_{2}}(n_2)$ as a linear combination of characteristic functions of the form $\mathds{1}_{\alpha\Z+b}(n_1)\mathds{1}_{\alpha\Z+b'}(n_2)$ for some $\alpha\in\Z[\tau_{d}]^{\times}$ and $b, b'\in\Z$,\footnote{Here we can take $\alpha$ to be in $\Z^{\times}$; later we need to convert $\alpha$ to a number in $\Z[\tau_{d}]^{\times}.$} which means it is enough to show that
\[
\lim_{N\to\infty}\E_{\NN(n_1+n_2\tau_{d}) \leq N}\mathds{1}_{\alpha\Z+b}(n_1)\mathds{1}_{\alpha\Z+b'}(n_2)f(n_1+n_2\tau_{d})=0.
\]
Again, up to multiples, this is the same as showing that
\begin{equation}\label{progression0eq}
\lim_{N\to\infty}\E_{\NN(\alpha u+\rho) \leq N} f(\alpha u+\rho)=0,
\end{equation}
 where $\rho=b+b'\tau_{d}$.
 If $\rho=0$, then (\ref{progression0eq}) holds since $f(\alpha u)=f(\alpha)f(u)$ and since (iii) holds for $\chi\equiv 1$. Now we assume that $\rho\neq 0$.
 Because $d\in\N$, one can factor $\rho$ into irreducibles (although not necessarily in a unique way), so factoring out the common factors, it suffices to show that for every principal ideal $\langle \alpha\rangle \subseteq \Z[\tau_{d}]$, and every $\rho \in (\Z[\tau_{d}]/\langle \alpha \rangle)^{\times}$, the average in \eqref{progression0eq} is 0. We now simply notice that, $\mathds{1}_{\langle \alpha \rangle +\rho}(u)$ can be written as a finite linear combination of Dirichlet characters of period $\langle \alpha\rangle$, since $\rho$ is now necessarily invertible in $\Z[\tau_{d}]/\langle\alpha\rangle$ by construction. Thus we have that (iii)$\Rightarrow$(i).

Now we show that (i)$\Rightarrow$(iii).
 Let $\chi$ be a non-trivial Dirichlet character with period $I$. We observe that   there exists some non-unit $n \in \mathcal{O}_K^{\times}$   such that $nx \in I$ for every $x \in \OK$. Thus, we may write $\lim_{N\to\infty}\E_{u\in\Z[\tau_{d}]^{\times}\colon \NN(u)\leq N}(f\chi)(u)$ as a linear combination of averages of the form
\[
\lim_{N \to \infty} \E_{\NN(nu+\rho) \leq N} \chi(nu+\rho)f(nu+\rho)=\lim_{N \to \infty} \E_{\NN(nu+\rho) \leq N} \chi(\rho)f(nu+\rho).
\]
So as before, by using the basis $\{1,\tau_{d}\}$ we may easily convert this  into a linear combination of expressions of the form \eqref{eq: aperiodicdef}, completing the proof.
\end{proof}

\begin{lemma}\label{dircharequivalence}
In Proposition \ref{prop: dircharequivalence}, $(i)$ is equivalent to $(iv)$.
\end{lemma}

\begin{proof}
We first show that (i)$\Rightarrow$(iv).
Let $\chi'$ be a modified Dirichlet character of period $I$. Then we may write $\chi'(u)=\chi(u)+\mathds{1}_{B}(u)$ for some  Dirichlet character $\chi$ of period $I$, where $B$ is the set of $u$ with $u \mod I\notin (\OK/I)^{\times}$. Note that the set of $u\in\OK$ for which $I\subseteq (u)$ is the set of $u$ for which $(\NN(I),\NN(u))\neq 1$, which is clearly the union of finitely many elements in $\mathcal{AP}[\tau_{d}]$.
So we may write $\chi'=\chi+\sum_{i=1}^{k}\mathds{1}_{P_{i}}$ for some $k\in\N\cup\{0\}$ and $P_{1},\dots,P_{k}\in \mathcal{AP}[\tau_{d}]$. Since Condition (i) implies Conditions (ii) and (iii), we have that 
\begin{eqnarray*}
\lim_{N\to\infty}\E_{u\in\mathcal{O}_{K}\colon \NN(u)\leq N}(f\chi')(u) & = & \lim_{N\to\infty}\E_{u\in\mathcal{O}_{K}\colon \NN(u)\leq N}(f\chi)(u)\\
& + & \sum_{i=1}^{k}\lim_{N\to\infty}\E_{u\in\mathcal{O}_{K}\colon \NN(u)\leq N}\mathds{1}_{P_{i}}(u)f(u)=0.
\end{eqnarray*}
So (iv) holds.

We now show that (iv)$\Rightarrow$(i). 
Similar to the argument in Lemma~\ref{lem: dircharequivalence},
it suffices to show that for every principal ideal $\langle \alpha\rangle \subseteq \Z[\tau_{d}]$, and every $\rho \in (\Z[\tau_{d}]/\langle \alpha \rangle)^{\times}$, the average in \eqref{progression0eq} is 0. Clearly \eqref{progression0eq} holds  when $\alpha=1$. Now assume that \eqref{progression0eq} holds  when $\NN(\alpha)\leq k$ for some $k\in\N$. Take $\alpha$ with $\NN(\alpha)\geq k+1$ and the additional property (that can be assumed without loss of generality) that if $\alpha=t\alpha'$ for some non-unit $t$, then $\NN(\alpha')\leq k$. Similar to the argument in Lemma~\ref{lem: dircharequivalence}, to show that \eqref{progression0eq} holds for this $\alpha$, it suffices to show that 
 \begin{equation}\nonumber
        \lim_{N\to\infty}\E_{u\in\Z[\tau_{d}]^{\times}\colon \NN(u)\leq N}(f\chi)(u)=0
        \end{equation}
for all Dirichlet characters $\chi$ of period $\langle\alpha\rangle$. We may rewrite  $\chi=\chi'-\mathds{1}_{B}$ for some modified Dirichlet character $\chi'$ of period $\langle \alpha\rangle$,  where $B$ is the set of $u$ with $u \mod \langle \alpha\rangle\notin (\OK/\langle \alpha\rangle)^{\times}$. It is clear that the set $B$ can be expressed as the disjoint union $B=\sqcup_{i=1}^{\ell}(u_{i}+\langle \alpha\rangle)$ for some $\ell\in\N$ and $u_{i}\in\OK$. Since (iv) holds, it suffices to show that 
 \begin{equation}\nonumber
\lim_{N\to\infty}\E_{u\in\Z[\tau_{d}]^{\times}\colon \NN(u)\leq N}\mathds{1}_{u_{i}+\langle \alpha\rangle}(u)f(u)=0
        \end{equation}
for all $1\leq i\leq \ell$. Since $u_{i}\notin (\OK/\langle \alpha\rangle)^{\times}$, there exists a non-unit $t\in\OK$ such that $u_{i}=tu'_{i}$ and $\alpha=t\alpha'$ for some $u'_{i},\alpha'\in\OK^{\times}$. So it suffices to show that
\begin{equation}\nonumber
\lim_{N\to\infty}\E_{u'\in\Z[\tau_{d}]^{\times}\colon \NN(u')\leq N}\mathds{1}_{u'_{i}+\langle \alpha'\rangle}(u')f(tu')=0.
        \end{equation} 
 Since $\NN(\alpha')<\NN(\alpha)$, the conclusion follows from the induction hypothesis and the fact that $f(tu')=f(t)f(u')$.
\end{proof}

 We are now in position to complete the proof of Proposition \ref{characterization}. 
 As usual, we let $K=\Q(\sqrt{d})$. Assume that $f\colon\mathcal{O}_K^{\times}\to\S$ is a completely multiplicative function and that $(I_{1},\dots,I_{k})$ is an ideal class representation. Let $\tilde{f}$ be an extension of $f$ (which exists by Proposition~\ref{prop: extensions}).
 Assume that $I_{i}^{d_{i}}=(x_{i})$ for some $x_{i}\in\mathcal{O}_K^{\times}$ and set $A=[d_{1}]\times \dots\times [d_{k}]$. For $\vec{i}=(i_{1},\dots,i_{k})\in A$, let $I_{\vec{i}}:=I_{1}^{i_{1}}\dots I_{k}^{i_{k}}$ and let $\mathfrak{I}_{\vec{i}}(K)$ denote the set of integer ideals in the same ideal class as $I_{\vec{i}}$.

   Since $f(x_{i})\neq 0$, we have that $f(\NN(I_{i}))\neq 0$.
Denote $$w_{\vec{i},N}:=\frac{\vert\{\mathfrak{v}\in \mathfrak{I}_{\vec{i}}(K)\colon \NN(\mathfrak{v})\leq N\}\vert}{\vert\{\mathfrak{v}\in \mathfrak{I}(K)\colon \NN(\mathfrak{v})\leq N\}\vert}.$$
 Then by  Lemma \ref{cl2},
    \begin{equation}\label{fweopf0}
    \begin{split}
        &\qquad \E_{\mathfrak{v}\in \mathfrak{I}(K), \NN(\mathfrak{v})\leq N}\tilde{f}(\mathfrak{v})
        \\&=\sum_{\vec{i}\in A} \E_{\mathfrak{v}\in \mathfrak{I}(K), \NN(\mathfrak{v})\leq N}\mathds{1}_{\mathfrak{v}\in \mathfrak{I}_{\vec{i}}(K)}\tilde{f}(\mathfrak{v})
        \\&=\sum_{\vec{i}\in A}w_{\vec{i},N} \E_{\frac{(u)}{(\NN(I_{\vec{i}}))}\in \mathfrak{Q}(K)\colon \frac{(u)}{(\NN(I_{\vec{i}}))}I_{\vec{i}}\in \mathfrak{I}(K), \NN(\frac{(u)}{(\NN(I_{\vec{i}}))}I_{\vec{i}})\leq N} \tilde{f}\left(\frac{(u)}{(\NN(I_{\vec{i}}))}I_{\vec{i}}\right)
        \\&=\sum_{\vec{i}\in A}w_{\vec{i},N}\tilde{f}(I_{\vec{i}})\cdot  \frac{1}{T}\E_{u\in \mathcal{O}_K^{\times}\colon \frac{(u)}{(\NN(I_{\vec{i}}))}I_{\vec{i}}\in \mathfrak{I}(K), \NN(\frac{(u)}{(\NN(I_{\vec{i}}))}I_{\vec{i}})\leq N} f(u)f(\NN(I_{\vec{i}}))^{-1}
        \\&=\sum_{\vec{i}\in A}w_{\vec{i},N}C_{\vec{i},N}(f)\tilde{f}(I_{1})^{i_{1}}\dots \tilde{f}(I_{k})^{i_{k}}, 
    \end{split}
\end{equation}
where  $T$ is the number of units of $\Q(\sqrt{-d})$ which is finite by Lemma \ref{l32}.
$$C_{\vec{i},N}(f):=\frac{1}{T}\E_{u\in \mathcal{O}_K^{\times}\colon \frac{(u)}{(\NN(I_{\vec{i}}))}I_{\vec{i}}\in \mathfrak{I}(K), \NN(\frac{(u)}{(\NN(I_{\vec{i}}))}I_{\vec{i}})\leq N} f(u)f(\NN(I_{\vec{i}}))^{-1},$$
which is independent of the choices of the extension $\tilde{f}$.

Suppose first that (i) holds.  Now let $\chi'$ be a modified Dirichlet character. 
We first claim that the average of $f\chi'$ along every arithmetic progression in $\mathcal{AP}[\tau_{d}]$ is 0. Since the indicator function of every set in $\mathcal{AP}[\tau_{d}]$ is a linear combination Dirichlet characters, it suffices to show that the average of $f\chi'\chi$ is 0 for all Dirichlet characters $\chi$. However, since $\chi'\chi$ is a Dirichlet character, the claim follows from the fact that (i)$\Rightarrow$(iii).

Since the set of $u$ for which  $\frac{(u)}{(\NN(I_{\vec{i}}))}I_{\vec{i}}\in \mathfrak{I}(K)$ is the disjoint union of finitely many elements in $\mathcal{AP}[\tau_{d}]$, it follows from (iv) (which is equivalent to (i)) that   $\lim_{N\to\infty}C_{\vec{i},N}(f\chi')=0$ for all $\vec{i}$. 

So for any extension $\widetilde{f\chi'}$ of $f\chi'$, it follows from \eqref{fweopf0} that 
$$\lim_{N\to\infty} \E_{\mathfrak{v}\in \mathfrak{I}(K), \NN(\mathfrak{v})\leq N}\widetilde{f\chi'}(\mathfrak{v})=0.$$
This implies (v).

Conversely, assume that (v) holds. We show that this implies (iv). Let $\chi'$ be a modified Dirichlet character. 
Let $\mathcal{F}$ denote the set of all extensions of $f\chi'$.
Then it follows from (v) that $$\lim_{N\to\infty} \sum_{g\in\mathcal{F}}\E_{\mathfrak{v}\in \mathfrak{I}(K), \NN(\mathfrak{v})\leq N}g(\mathfrak{v})=0.$$
On the other hand, by \eqref{fweopf0}, we have that 
   \begin{eqnarray*}
 \sum_{g\in\mathcal{F}}\E_{\mathfrak{v}\in \mathfrak{I}(K), \NN(\mathfrak{v})\leq N}g(\mathfrak{v}) 
    & = & \sum_{g\in\mathcal{F}}\sum_{\vec{i}\in A}w_{\vec{i},N}C_{\vec{i},N}(f\chi')g(I_{1})^{i_{1}}\dots g(I_{k})^{i_{k}}
        \\& = &\sum_{a_{j}^{d_{j}}=(f\chi')(x_{j}), 1\leq j\leq k}\sum_{\vec{i}\in A}w_{\vec{i},N}C_{\vec{i},N}(f\chi')a_{1}^{i_{1}}\dots a_{k}^{i_{k}}
        \\& = &\sum_{\vec{i}\in A}w_{\vec{i},N}C_{\vec{i},N}(f\chi')\cdot\sum_{a_{j}^{d_{j}}=(f\chi')(x_{j}), 1\leq j\leq k}a_{1}^{i_{1}}\dots a_{k}^{i_{k}}
        \\& = & w_{\vec{0},N}C_{\vec{0},N}(f\chi').
 \end{eqnarray*}
So,
$$\lim_{N\to\infty}w_{\vec{0},N}C_{\vec{0},N}(f\chi')
        =\lim_{N\to\infty} \sum_{g\in\mathcal{F}}\E_{\mathfrak{v}\in \mathfrak{I}(K), \NN(\mathfrak{v})\leq N}g(\mathfrak{v})=0.$$
        On the other hand, it follows from \cite[Theorem~11.1.5]{murtyesmondebook} that
        $\lim_{N\to\infty}w_{\vec{0},N}=\frac{1}{|A|}$. So
        $\lim_{N\to\infty}C_{\vec{0},N}(f\chi')=0$ by taking their quotients. So (iv) holds for such $\chi'$ and 
we are done.

We are now ready to prove Theorem \ref{characterization}:

 
\begin{proof}[Proof of Theorem \ref{characterization}]  
Suppose that $f$ is not aperiodic. By  Lemma \ref{tv}, we have that $f(\e)=1$ for all units $\e$. By Proposition \ref{prop: dircharequivalence}, 
there exists a modified Dirichlet character $\chi\colon \Z[\tau_{d}]^{\times}\to\S$ and some 
         extension $\widetilde{f\chi}$ of $f\chi$ 
        such that the following fails to be true
        $$\lim_{N\to\infty}\E_{\mathfrak{v}\in \mathfrak{I}(\Q(\sqrt{-d}))\colon \NN(\mathfrak{v})\leq N}\widetilde{f\chi}(\mathfrak{v})=0.$$
 By Theorem \ref{C22intro}, we must have that $\D(\widetilde{f\chi},\NN(\mf u)^{i\tau})<\infty$ for some $\tau\in\R$.
 It is not hard to see that this implies that $f$ is  pretentious.
 %
\end{proof}
A consequence of Theorem~\ref{characterization} above is that it allows us to easily obtain the following measurability result in $\mathcal{M}$:
\begin{lemma}\label{measurability}
The set of pretentious completely multiplicative functions $\mathcal{M}_p$ is Borel.
\end{lemma}
\begin{proof}
By Theorem~\ref{characterization} and the definition of aperiodic functions, we may write 
\[
\mathcal{M}_p = \bigcup_{P\in\mathcal{AP}[\tau_{d}]} \left\{ f \in \mathcal{M} : \limsup_{N \to \infty} \left| \E_{u \in P, \NN(u) \leq N } f(u)\right|>0\right\}.
\]
Since $\mathcal{AP}[\tau_{d}]$ is countable, the displayed union is countable. It now follows by standard methods that $\mathcal{M}_p$ is a countable union of subsets that are each Borel measurable, since they are achieved as the nullset of $\limsup$ of sequences of continuous functions on $\mathcal{M}$, and the set $\mathcal{M}$ comes equipped with the topology of pointwise convergence.
\end{proof}

\section{Proof of the Tur\'an-Kubilius inequality and applications}\label{SECconcest}
In this section we prove the  Tur\'an-Kubilius inequality Theorem \ref{TK} and then use it to obtain an analog concentration estimates of \cite[Proposition~2.5]{FKM1}.
%
%
%
Throughout this section, when $K=\Q(\sqrt{-d})$ for some squarefree $d\in\N$,
let $\gamma_{K}:=\lim_{N \to \infty} \frac{1}{N}\sum_{\NN(u) \leq N} 1$
 be the constant given by Corollary~\ref{lem: bounds for number of ideals of norm at most x}.
 We will need a counting lemma, but in order to properly show it, we must first have a certain estimate for $C$-Lipschitz regions of $\R^2$, the definition of which we recall next.

\subsection{Proof of the Tur\'an-Kubilius inequality} In this section we prove Theorem \ref{TK}.

First we need the definition of $C$-Lipschitz regions of $\R^2$.
\begin{definition}
Let $S\subseteq \mathbb{R}^{2}$ be the region given in polar coordinates by
     $$S:=\{(\rho\cos\theta,\rho\sin\theta)\colon \theta\in [0,2\pi), 0\leq \rho\leq r(\theta)\}$$
     for some continuous $2\pi$-periodic map $r\colon \R\to \R_{\geq 0}$, i.e., with $r(\cdot+2\pi)=r(\cdot)$. We say that $S$ is \emph{$C$-Lipschitz} if
     $$r(\theta)\leq C\text{ and } \vert r(\theta)-r(\theta')\vert\leq C\vert\theta-\theta'\vert$$
      for all $\theta,\theta'\in\R$.
\end{definition}
We can now give a counting lemma, based on \cite[Lemma~10.1]{kubiliusbook}.
 \begin{lemma}\label{dafp}
 Let $C\geq 1$.
 For any  $C$-Lipschitz region $S\subseteq \mathbb{R}^{2}$, any $v\in\R^{2}$ and any $N\in\R_{+}, N\geq 1$,  we have that
$$\left|\left| (S_{N}+v)\cap \Z^{2}\right|-N^{2}\cdot m(S)\right|\leq O(CN),$$
where $S_{N}+v:=\{x\in\R^{2}\colon \frac{1}{N}(x-v)\in S\}$ and $m(S)$ is the area of $S$.
 \end{lemma}
 \begin{proof}

     Since $\lceil C\rceil\leq 2C$, we may assume without loss of generality that $C\in\N$.
      Let $B_{N}$ be the union of unit boxes in $\R^{2}$ whose lower left corner belongs is a lattice point belonging to $(S_{N}+v)\cap \Z^{2}$. Then, by construction, we ensured that $m(B_{N})=\vert(S_{N}+v)\cap \Z^{2}\vert$. Fix $N$. Divide $[0,2\pi)$ into intervals $I_{i}=[\frac{2\pi(i-1)}{CN},\frac{2\pi i}{CN})$
for $1\leq i\leq CN$, let $r_{i,\inf}:=\inf_{\theta\in I_{i}}r(\theta)$ and $r_{i,\sup}:=\sup_{\theta\in I_{i}}r(\theta)$, and define 
$$D_{N,i}:=\{(\rho\cos \theta,\rho\sin\theta)\colon \theta\in I_{i}, \rho\in [r_{i,\inf}N,r_{i,\sup}N]\}$$
and $D'_{N,i}$ be the set of points whose distance to $D_{N,i}$ is at most 2.

We first claim that $m(D'_{N,i})=O(1)$.
Fix $N$ and $i$. Let $$E_{1}:=\{(\rho\cos \theta,\rho\sin\theta)\colon \theta\in I_{i}, \rho\in [r_{i,\inf}N-10,r_{i,\sup}N+10]\},$$
$E_{2}$ be the rectangle adjacent to $E_{1}$ with one of its edge being $$\left\{(\rho\cos \theta,\rho\sin\theta)\colon \theta=\frac{2\pi (i-1)}{CN}, \rho\in [r_{i,\inf}N-10,r_{i,\sup}N+10]\right\}$$  and the other edge of length 10 pointing out of $E_{1}$, and $E_{3}$ be the rectangle adjacent to $E_{1}$ with one of its edge being $$\{(\rho\cos \theta,\rho\sin\theta)\colon \theta=\frac{2\pi i}{CN}, \rho\in [r_{i,\inf}N-10,r_{i,\sup}N+10]\}$$  and the other edge of length 10 pointing out of $E_{1}$. Then $D'_{N,i}\subseteq E_{1}\cup E_{2}\cup E_{3}$.

Since
\begin{equation*}
    \begin{split}
        \qquad &
        m(E_{1})=\frac{2\pi}{CN}((r_{i,\sup}+10)^{2}-(r_{i,\inf}-10)^{2})
        =\frac{2\pi}{CN}(r_{i,\sup}+r_{i,\inf})(r_{i,\sup}-r_{i,\inf}+20)
       \\& = \frac{2\pi}{CN}2CN(CN\cdot\frac{2\pi}{CN}+20)=O(1)
    \end{split}
\end{equation*}
and
$$m(E_{i})=10 (r_{i,\sup}-r_{i,\inf}+20)\leq 10 (CN\cdot\frac{2\pi}{CN}+20)=O(1)$$
for $i=2,3$. The claim follows.

\

Note that if $x\in B_{N}\Delta (S_{N}+v)$. Then there exists $1\leq i\leq CN$ such that the distance between $x$ and the boundary
$$\partial S_{N,i}:=\{(Nr(\theta)\cos\theta,Nr(\theta)\sin\theta)\colon \theta\in I_{i}\}$$ is at most 2. Since $\partial  S_{N,i}\subseteq D_{N,i}$, we have that the distance between $x$ and $D_{N,i}$ is at most 2 and thus $x\in\cup_{i=1}^{CN}D'_{N,i}$.
Thus, by the previous claim we obtain the desired bound
\begin{equation*}
    \begin{split}
        m(B_{N}\Delta (S_{N}+v))\leq \sum_{i=1}^{CN}m(D'_{N,i})=O(CN)
    \end{split}
\end{equation*}
and we are done. 
 \end{proof}
 

\begin{lemma}\label{countingCRT}
Let $K=\Q(\sqrt{-d})$ for some squarefree $d\in\N$, $k, l \in \N, a,Q\in\mathcal{O}_K^{\times}$, and $I$ be an ideal of $\OK$ such that $(Q)$ is coprime to $I$. Let $B_N:=\{ u\in\mathcal{O}_{K} : \NN(u) \leq N\}$. There exists $C>0$ depending only on $d$   such that for any $N>0$, we have that 
\begin{equation}\nonumber
    \Bigl\vert \frac{\gamma_{K} N}{\NN(I)}-\vert(QB_N+a)\cap I|\Bigr\vert\leq C\sqrt{N/\NN(I)}.
\end{equation}
\end{lemma}

\begin{proof}

 Given that $I$ is an ideal in $K$, an imaginary extension or $\Q$, it follows that $I$ is a free $\Z$-module of rank 2. Thus, we may write $I$ as 
     $$I=\{m\alpha+n\beta\colon m,n\in\Z\}$$
     for some $\alpha=a_{1}+b_{1}\tau_{d},\beta=a_{2}+b_{2}\tau_{d}\in\Z[\tau_{d}]$. As a consequence of Minkowski's bound (see \cite[Theorems 6.6 and 7.4]{neukirch99}), we see that the fundamental parallelepiped that generates the ideal lattice $I$ has volume proportional to the norm of the ideal, where the constant only depends on the number field $K$. Thus, we may take generators $\alpha, \beta$ whose norm is $O(\NN(I))$. 
     This means that we may further assume that $\vert a_{i}\vert, \vert b_{i}\vert\leq C\sqrt{\NN(I)}$ for some universal constant $C$ depending only on $d$. 
     Since $(Q)$ is coprime to $I$, there exist unique $(m_{0},n_{0})\in\{0,\dots,Q-1\}^{2}$ such that $Q\vert m\alpha+n\beta-a$ if and only if $m\equiv m_{0} \mod Q$ and $n\equiv n_{0} \mod Q$.
     Therefore, the set $(QB_{N}+a)\cap I$ consists of elements of the form $(Qm+m_{0})\alpha+(Qn+n_{0})\beta$ with $m,n\in\Z$ such that $\NN((Qm+m_{0})\alpha+(Qn+n_{0})\beta-a)\leq Q^{2}N$.
     

      
      Let $S$ be the set of $(m,n)\in\R^{2}$ such that $\NN(m\alpha+n\beta)\leq 1$ (where $\NN$ is the canonical extension of the norm on $\Q(\sqrt{-d})$ to $\C$). Then 
     $$\vert (QB_{N}+a)\cap I\vert=\Bigl\vert (S_{\sqrt{N}}+\frac{1}{Q}(a-(m_{0}\alpha+n_{0}\beta)))\cap\Z^{2}\Bigr\vert.$$ Let $T\colon \Q^{2}\to\Q^{2}$ be the bijective linear transformation given by 
      $$T(m,n):=\iota(m\alpha+n\beta),$$
      where $\iota\colon\Q(\sqrt{-d})\to\Q^{2}$ is the natural map given by $\iota(m+n\tau_{d}):=(m,n)$.
      Let $S':=\{(m,n)\in\R^{2}\colon \NN(m+\tau_{d} n)\leq 1\}$. Clearly $S=T^{-1}S'$.
      Thus 
      $$m(S)=\NN(I)^{-1}m(S')=\gamma_{K}/\NN(I).$$

       Using polar coordinates for ellipses, it is not hard to see that the Lipschitz consant of an ellipse is bounded by its diameter. So $S'$ is $O_{d}(1)$-Lipschitz. Since $\vert a_{i}\vert, \vert b_{i}\vert\leq O_{d}(\sqrt{\NN(I)})$, we have that $S$ is $O_{d}(\sqrt{\NN(I)})$-Lipschitz. By Lemma~\ref{dafp}, we have 
      \[\left\vert\vert (QB_{N}+a)\cap I\vert-\frac{\gamma_{K}N}{\NN(I)}\right\vert=\Bigl\vert\vert (S_{\sqrt{N}}+\frac{1}{Q}(a-(m_{0}\alpha+n_{0}\beta)))\cap \Z^{2}\vert-N\cdot m(S)\Bigr\vert\leq O_{d}\left(\sqrt{\frac{N}{\NN(I)}}\right),\]
      as was to be shown.
\end{proof}

\begin{proof}[Proof of Theorem \ref{TK}]
We first assume that $h$ is non-negative.
 Note that
\begin{equation}\label{firstTK}
    \begin{split}
        &\quad \E_{\NN(u) \leq N}h((Qu+a))
        =\frac{1}{\gamma_{K} N}\sum_{\NN(\frac{u-a}{Q})\leq N, u\equiv a \mod Q}h((u))
        \\&\quad\quad\quad\quad\quad\quad\quad\quad\quad\quad\;\;=\frac{1}{\gamma_{K} N}\sum_{\NN(\frac{u-a}{Q})\leq N, u\equiv a \mod Q}\sum_{\mf p^{k}|| u} h(\mf p^{k})
        \\&\quad\quad\quad\quad\quad\quad\quad\quad\quad\quad\;\;=\frac{1}{\gamma_{K} N}\sum_{\NN(\frac{u-a}{Q})\leq N, u\equiv a \mod Q}\sum_{\mf p^{k}|| u, \mf p\nmid Q} h(\mf p^{k}).
    \end{split}
\end{equation}
%
 %
 Note that if $\mf p^{k}\vert u$ for some $u$ with $\NN(\frac{u-a}{Q})\leq N$, then
$$\NN(\mf p^{k})\leq \NN(u)\leq 2(\NN(u-a)+\NN(a))\leq 2(\NN(Q)N+\NN(Q))=N'.$$ 
By Lemma~\ref{countingCRT},
\begin{equation}\label{firstTK2}
    \begin{split}
        &\quad 
        \frac{1}{\gamma_{K} N}\sum_{\NN(\frac{u-a}{Q})\leq N, u\equiv a \mod Q}\sum_{\mf p^{k}|| u, \mf p\nmid Q} h(\mf p^{k})
        \\&=\frac{1}{\gamma_{K} N}\sum_{\NN(\mf p^{k})\leq N', \mf p\nmid Q}h(\mf p^{k})\sum_{\NN(\frac{u-a}{Q})\leq N, \mf p^{k}||u, u\equiv a \mod Q} 1
    \\& =\frac{1}{\gamma_{K} N}\sum_{\NN(\mf p^{k})\leq N', \mf p\nmid Q}h(\mf p^{k})\left(\sum_{\NN(\frac{u-a}{Q})\leq N, \mf p^{k}|u, u\equiv a \mod Q} 1-\sum_{\NN(\frac{u-a}{Q})\leq N, \mf p^{k+1}|u, u\equiv a \mod Q} 1\right)
    \\&=\sum_{\NN(\mf p^{k})\leq N', \mf p\nmid Q}h(\mf p^{k})\left((\NN(\mf p)^{-k}-\NN(\mf p)^{-(k+1)})+O\left(\frac{1}{\sqrt{N\NN(\mf p^{k})}}\right)\right).
    \end{split}
\end{equation}
for some $C>0$ depending only on $K$ and $Q$. Therefore,
combining (\ref{firstTK}) and (\ref{firstTK2}), we have that
\begin{equation}\label{firstTK0}
    \begin{split}
       &\qquad \left| A_{Q,N}\cdot\left(\E_{u\in\OK\colon\NN(u) \leq N}h((Qu+a))-A_{N'}\right)\right|
       \\&\ll \frac{1}{\sqrt{N}}\sum_{\NN(\mf p^{k})\leq N', \mf p\nmid Q}\vert h(\mf p^{k})\vert \cdot \NN(\mf p^{k})^{-1/2}\cdot \sum_{\NN(\mf p^{k})\leq N', \mf p\nmid Q}\vert h(\mf p^{k})\vert\cdot \NN(\mf p^{k})^{-1}
        \\&\leq \frac{1}{\sqrt{N}}\sum_{\NN(\mf p^{k})\leq N', \NN(\mf p)>M}\vert h(\mf p^{k})\vert\cdot \NN(\mf p^{k})^{-1/2}\cdot \sum_{\NN(\mf p^{k})\leq N', \NN(\mf p)>M}\vert h(\mf p^{k})\vert\cdot \NN(\mf p^{k})^{-1}
          \\&\leq \frac{1}{\sqrt{N}}\left(B_{N'}\cdot\sum_{\NN(\mf p^{k})\leq N'} 1\right)^{1/2}\cdot\left(B_{N'}\cdot\sum_{\NN(\mf p^{k})\leq N'}\NN(\mf p^{k})^{-1}\right)^{1/2}
          \\&=B_{N'}\cdot \frac{1}{\sqrt{N}}\left(\sum_{\NN(\mf p^{k}),\NN(\mf q^{\ell})\leq N'}\NN(\mf q^{\ell})^{-1}\right)^{1/2}.
          \end{split}
\end{equation}

Now we estimate the last term of (\ref{ltk}). We have
\begin{equation}\label{secondTK}
    \begin{split}
        &\quad \E_{\NN(u) \leq N}\vert h((Qu+a))\vert^{2}
        =\frac{1}{\gamma_{K} N}\sum_{\NN(\frac{u-a}{Q})\leq N, u\equiv a \mod Q}\vert h((u))\vert^{2}
        \\&=\frac{1}{\gamma_{K} N}\sum_{\NN(\frac{u-a}{Q})\leq N, u\equiv a \mod Q}\left| \sum_{\mf p^{k}|| u} h(\mf p^{k})\right|^{2}
        \\&=\frac{1}{\gamma_{K} N}\sum_{\NN(\frac{u-a}{Q})\leq N, u\equiv a \mod Q}\left(\sum_{\mf p^{k}|| u} \vert h(\mf p^{k})\vert^{2}+\sum_{\mf p^{k},\mf q^{\ell}|| u, \mf p\neq \mf q} h(\mf p^{k})h(\mf q^{\ell})\right)
          \\&=\frac{1}{\gamma_{K} N}\sum_{\NN(\frac{u-a}{Q})\leq N, u\equiv a \mod Q}\left(\sum_{\mf p^{k}|| u, \mf p\nmid Q} \vert h(\mf p^{k})\vert^{2}+\sum_{\mf p^{k},\mf q^{\ell}|| u, \mf p\neq \mf q, \mf p,\mf q\nmid Q} h(\mf p^{k})h(\mf q^{\ell})\right).
    \end{split}
\end{equation}
The first term in the last line of \eqref{secondTK} can be dealt with in the same way as in  \eqref{firstTK2} 
(replacing $h$ by $h^{2}$), namely 
\begin{equation}\label{secondTK2}
    \begin{split}
        &\quad 
        \frac{1}{\gamma_{K} N}\sum_{\NN(\frac{u-a}{Q})\leq N, u\equiv a \mod Q}\sum_{\mf p^{k}|| u, \mf p\nmid Q} h(\mf p^{k})^{2}
     \\&=\frac{1}{\gamma_{K} N}\sum_{\NN(\mf p^{k})\leq N', \mf p\nmid Q}h(\mf p^{k})^{2}\sum_{\NN(\frac{u-a}{Q})\leq N, \mf p^{k}||u, u\equiv a \mod Q} 1
     \\&=\sum_{\NN(\mf p^{k})\leq N', \mf p\nmid Q}h(\mf p^{k})^{2}\NN(\mf p)^{-k}+O(1)\cdot\frac{1}{\sqrt{N}}\sum_{\NN(\mf p^{k})\leq N', \mf p\nmid Q}h(\mf p^{k})^{2}\cdot \NN(\mf p^{k})^{-1/2}
      \\&\leq\sum_{\NN(\mf p^{k})\leq N', \mf p\nmid Q}h(\mf p^{k})^{2}\NN(\mf p)^{-k}+O(C_{N}).
    \end{split}
\end{equation}

For the second term in \eqref{secondTK}, similarly to (\ref{firstTK}), we also swap the order of the sums and rewrite it as
\begin{equation}\nonumber
    \begin{split}
        &\quad \frac{1}{\gamma_{K} N}\sum_{\NN(\mf p^{k}\mf q^{\ell}) \leq N', \mf p\neq \mf q,\mf p,\mf q\nmid Q} h(\mf p^{k})h(\mf q^{\ell})\sum_{\NN(\frac{u-a}{Q})\leq N, u\equiv a \mod Q, \mf p^{k}||u, \mf q^{\ell}||u}1.
    \end{split}
\end{equation}
As was the case in \eqref{firstTK2}, we have that
\begin{equation}\label{secondTK3}
    \begin{split}
        &\quad \frac{1}{\gamma_{K} N}\sum_{\NN(\mf p^{k}\mf q^{\ell}) \leq N', \mf p\neq \mf q,\mf p,\mf q\nmid Q} h(\mf p^{k})h(\mf q^{\ell})\sum_{\NN(\frac{u-a}{Q})\leq N, u\equiv a \mod Q, \mf p^{k}||u, \mf q^{\ell}||u}1
        \\&=\sum_{\NN(\mf p^{k}\mf q^{\ell})\leq N', \mf p\neq \mf q, \mf p,\mf q\nmid Q} h(\mf p^{k})h(\mf q^{\ell})\left(\NN(\mf p^k)^{-1}\NN(\mf q^l)^{-1}\left(1-\frac{1}{\NN(\mf p)}\right) \left(1-\frac{1}{\NN(\mf q)}\right)\right. \\&\quad\quad\quad\quad\quad\quad\quad\quad\quad\quad\quad\quad\quad\quad\quad\quad\quad\quad\quad\quad\quad\quad\quad\quad\quad\left.+O\left(\frac{1}{\sqrt{N\NN(\mf p^{k}\mf q^{\ell})}}\right)\right)
        \\&\leq A_{Q,N'}^{2}+\frac{C}{\sqrt{N}}\sum_{\NN(\mf p^{k}\mf q^{\ell})\leq N', \mf p\neq \mf q, \NN(\mf p),\NN(\mf q)>M} h(\mf p^{k})h(\mf q^{\ell})\NN(\mf p^{k})^{-1/2}\NN(\mf q^{\ell})^{-1/2}     
        \\&\leq A_{Q,N'}^{2}+B_{N'}\cdot\frac{C}{\sqrt{N}}\left(\sum_{\NN(\mf p^{k}\mf q^{\ell})\leq N', \mf p\neq \mf q} 1\right)^{1/2}.
            \end{split}
\end{equation}

Combining (\ref{secondTK}), (\ref{secondTK2}) and (\ref{secondTK3}), we have that
\begin{equation}\label{secondTK0}
    \begin{split}
        &\quad \E_{\NN(u) \leq N}\vert h((Qu+a))\vert^{2}
        \\&\leq 
        B_{N'}+A_{Q,N'}^{2}+O(C_{N})+\frac{C}{N}\sum_{\NN(\mf p^{k})\leq N', \NN(\mf p)>M}h(\mf p^{k})^{2}
        \\&
        \quad\quad\quad\quad\quad\quad\quad\quad\quad+B_{N'}\cdot\frac{C}{N}\left(\sum_{\NN(\mf p^{k}\mf q^{\ell})\leq N', \mf p\neq \mf q} \NN(\mf p^{k})\cdot \NN(\mf q^{\ell})\right)^{1/2}.
    \end{split}
\end{equation}   

Finally, by Part (i) of Lemma \ref{l2}, 
\begin{equation}\label{snnn2}
    \begin{split}
        &\quad A_{Q,N'}-A_{Q,N}
        =\sum_{N<\NN(\mf p^{k})\leq N',\mf p\nmid Q}h(\mf p^{k})\NN(\mf p)^{-k}(1-\NN(\mf p)^{-1})
        \\&\leq \frac{2}{N}\sum_{N<\NN(\mf p^{k})\leq N',\mf p\nmid Q}h(\mf p^{k})
        \leq \frac{2}{N}\sum_{N<\NN(\mf p^{k})\leq N',\mf \NN(\mf p)>M}h(\mf p^{k})
        \\&\leq B_{N'}^{1/2}\cdot \frac{2}{N}\left(\sum_{N<\NN(\mf p^{k})\leq N'}\NN(\mf p^{k})\right)^{1/2}
        \leq B_{N'}^{1/2}\cdot \frac{2}{N}\left(N'\sum_{N<\NN(\mf p^{k})\leq N'}1\right)^{1/2}
        \\&\ll B_{N'}^{1/2}\cdot \frac{2}{N}\left(\frac{{N'}^{2}}{\log N'}\right)^{1/2}=o_{Q;N\to\infty}(B_{N'}^{1/2}).
        \end{split}
\end{equation} 
Combining 
(\ref{firstTK0}), (\ref{secondTK0}) and (\ref{snnn2}), we have
%
%
 \begin{equation}\label{snnn}
    \begin{split}
        &\quad \E_{\NN(u) \leq N}\left| h((Qu+a))-A_{Q,N}\right|^{2}  
        \\&=(A_{Q,N}-A_{Q,N'})^{2}-2A_Q,{N}\left(\E_{\NN(u) \leq N}h((Qu+a))-A_{Q,N'}\right)\\&\quad\quad\quad\quad\quad\quad\quad\quad\quad\quad\quad\quad\quad\quad\quad\quad\quad+\E_{\NN(u) \leq N}\vert h((Qu+a))\vert^{2}-A_{Q,N'}^{2}
        \\&\leq B_{N'}\cdot\left(1+C\left(\frac{1}{N}\sum_{\NN(\mf p^{k}),\NN(\mf q^{\ell})\leq N'}\NN(\mf q^{\ell})^{-1}\right)^{1/2}\right.\\&\quad\quad\quad\quad\quad\quad\quad\quad\quad\left.+C\left(\frac{1}{N}\sum_{\NN(\mf p^{k}\mf q^{\ell})\leq N', \mf p\neq \mf q} 1\right)^{1/2}+o_{Q;N\to\infty}(1)\right)+O(C_{N})
        \\&=B_{N'}(
        1+o_{Q;N\to\infty}(1))+O(C_{N}),
        \end{split}
\end{equation}
where the last inequality follows from 
Lemma \ref{l2}.

When $h$ is real valued, then we let $h_{\pm}$ be the additive function defined by $h_{\pm}(\mf p^{k}):=\max\{\pm h(\mf p^{k}),0\}$. Then 
$$\left| h((Qu+a))-A_{Q,N}\right|^{2}\leq 2\left| h_{+}((Qu+a))-A_{Q,N,+}\right|^{2}+2\left| h_{-}((Qu+a))-A_{Q,N,-}\right|^{2},$$
$B_{N'}=B_{N',+}+B_{N',-}$ and $C_{N'}=C_{N',+}+C_{N',-}$, where $A_{N,\pm}, B_{N,\pm}$ and $C_{N,\pm}$ are defined similar to $A_{N}, B_{N}$ and $C_{N}$ but with $h$ replaced by $h_{\pm}$. So it follows from (\ref{snnn}) that (\ref{ltk}) holds when  $h$ is real valued.

When $h$ is complex valued, (\ref{ltk}) also holds by considering the real and imaginary part separately. We are done.
\end{proof}

 \subsection{The concentration estimate}
In this section we state and prove the concentration estimate we need.
  Let 
    \begin{equation}\label{folnerMult}
    \Phi_{M}:=\left\{\prod_{\NN(\mathfrak{p})\leq M}\mathfrak{p}^{a_{\mathfrak{p}}}\colon M<a_{\mathfrak{p}}\leq 2M, \prod_{\NN(\mathfrak{p})\leq M}\mathfrak{p}^{a_{\mathfrak{p}}} \text{ is principal} \right\}.
       \end{equation} 
    Note that since we can identify principal ideals with elements of $\mathcal{O}_K$ by using their generators, the subsets $\Phi_M$ can be seen as subsets of $\mathcal{O}_K$. 

\begin{lemma}
    The family of sets $(\Phi_{M})_{M}$ forms a F\o lner sequence in $\mathcal{O}_K$.
\end{lemma}
\begin{proof}
    let $\{\mathfrak{p}_{1},\dots,\mf p_k, \dots\}$ be an ordering of all the prime ideals with non-decreasing norms, and suppose that the set of $\mathfrak{p}$ with $\NN(\mathfrak{p})\leq M$ is $\mathfrak{p}_{1},\dots,\mathfrak{p}_{b_{M}}$. 
Let $G$ denote the ideal class group of $K$. For any ideal $I$, let $v(I)\in G$ denote the ideal class containing $I$.  




Let $u \in \mathcal{O}_K$, and suppose that $(u)=\prod_{i=1}^m \mf p_i^{c_i}$.
    Let $M$ be large so that $m\leq b_{M}$. 
z5   
Since $v((u))=\prod_{i=1}^{b_{K}}v(\mathfrak{p}_{i})^{c_{i}}=e_{G}$, we have
\begin{equation*}
    \begin{split}
     &   u\cdot\Phi_{M}=\left\{\prod_{i=1}^{b_{M}}\mathfrak{p}^{a_{i}+c_{i}}_{i}\colon M<a_{i}\leq 2M, \prod_{i=1}^{b_{M}}v(\mathfrak{p}_{i})^{a_{i}+c_{i}}=e_{G}\right\}
     \\&\quad\;\;\;\;\quad=\left\{\prod_{i=1}^{b_{M}}\mathfrak{p}^{a'_{i}}_{i}\colon M+c_{i}<a'_{i}\leq 2M+c_{i}, \prod_{i=1}^{b_{M}}v(\mathfrak{p}_{i})^{a'_{i}}=e_{G}\right\},
    \end{split}
\end{equation*}
   where $c_{i}=0$ for $i>m$. 
   Then $\Phi_{M}$ is nonempty as its cardinality is approximately $M^{b_{M}}/\vert G\vert$, and we have that
    $$\frac{\vert\Phi_{M}\Delta (u\cdot \Phi_{M})\vert}{\vert\Phi_{M}\vert}\lesssim\frac{\vert G\vert c^{m}(M+c)^{b_{M}-1}}{M^{b_{M}}}\to 0$$
    as $M\to\infty$,  where $c=2\cdot\max_{1\leq i\leq m}c_{i}$. So $(\Phi_{M})_M$ is a F\o lner sequence. 
\end{proof}

Let $f,g\colon\mf  I(K)\to\C$ be multiplicative functions. The truncated distance between $f$ and $g$ is given by:
$$\D(f,g;M_{1},M_{2}):=\sum_{M_{1}<\NN(\mf p)\leq M_{2}}\frac{1}{\NN(\mf p)}(1-\text{Re}(f(\mf p)\overline{g}(\mf p))).$$


We are now ready to state the main concentration estimate of the paper:

\begin{proposition}\label{p25}
    Let $K=\Q(\sqrt{-d})$ for some squarefree $d\in\N$ and $f\colon \mathcal{O}_K^{\times}\to\S$ be a multiplicative function that is trivial on units.\footnote{We have already discussed in Lemma~\ref{tv} that this can be assumed without loss of generality.}
    Suppose that there exist a modified Dirichlet character $\chi$ of period $I$, some $\tau\in\R$ and  some extension $\tilde{f'}$ of the function 
    $$f'(u):=f(u)\overline{\chi}(u)\NN(u)^{-i\tau}$$
    with
    $\D(\tilde{f'}, 1)<\infty$. 
  Let also $a \in \mathcal{O}_K^{\times}$ with $\NN(a) \leq \NN(Q)$ and $(a)$ being coprime to $(Q)$.
  %
   Suppose that $M$ is large enough so that $\Phi_M \subseteq I$. 
   We have 
\begin{equation}\nonumber
    \begin{split}
       & \limsup_{N\to\infty}\max_{Q\in \Phi_{M}, a\in\mathcal{O}_K^{\times},\atop \NN(a)\leq \NN(Q), (a,Q)=1}\E_{ \NN(u)\leq N}\vert f(Qu+a)-\chi(a)\NN(Qu)^{i\tau}\exp(F(\tilde{f'},M,N;Q))\vert
     \\&\quad\quad\quad\quad\quad\quad\quad\quad\quad\quad\quad\quad\quad\quad\quad\quad\quad\quad\quad\quad\quad\quad\quad\ll \D(\tilde{f'},1;M,\infty)+M^{-1/2},
    \end{split}
\end{equation}
where  the implicit constant depends only on $K$ and
$$F(\tilde{f'},M,N;Q):=\sum_{M<\NN(\mf p)\leq N, \mf p \nmid Q}\frac{1}{\NN(\mf p)}(1-\text{Re}(\tilde{f'}(\mf p))).$$
 \end{proposition}

 
\begin{proof} The proof is based on \cite[Lemma~2.5]{KMPT}. 
Note that if $Q \in I$, then
    $$f(Qu+a)=f'(Qu+a)\chi(a)\NN(Qu+a)^{i\tau}=f'(Qu+a)\chi(a)\NN(Qu)^{i\tau}+O_{\tau}(\log (\NN(1+u^{-1})).$$
    because $|e^{ix}-e^{iy}| \leq |x-y|$ for real numbers $x, y$ and $\ln \frac{\NN(Qu+a)}{\NN(Qu)}$ converges to 0 as $\NN(u) \to \infty$. 
    On the other hand, since  $\log \NN(1+u^{-1}) \to 0$ as $u \to \infty$ in the sense of leaving compact sets, we have that 
    $\lim_{N\to\infty}\E_{\NN(u) \leq N}\log \NN(1+u^{-1})=0$.
  So it suffices to show that 
\begin{equation}
    \begin{split}
       & \limsup_{N\to\infty}\max_{Q\in \Phi_{M}, a\in\mathcal{O}_K^{\times},\atop \NN(a)\leq \NN(Q), (a,Q)=1}\E_{ \NN(u)\leq N}\vert f'(Qu+a)-\exp(F(\tilde{f'},M,N;Q))\vert
     \\&\quad\quad\quad\quad\quad\quad\quad\quad\quad\quad\quad\quad\quad\quad\quad\quad\quad\quad\quad\quad\quad\quad\quad\ll \D(\tilde{f'},1;M,\infty)+M^{-1/2},
    \end{split}
\end{equation}

Let $h: \mf I(K)\to\C$ be the additive function whose values on powers of prime ideals are given by $h(\mf p^{k}):=\tilde{f'}(\mf p^{k})-1$. Since $z=e^{z-1}+O(\vert z-1\vert^{2})$, we have
$$f'(Qu+a)=\prod_{\mf p^{k}\Vert Qu+a, \NN(\mf p)>M} \tilde{f'}(\mf p^{k})=\prod_{\mf p^{k}\Vert Qu+a, \NN(\mf p)>M}\Bigl(\exp(h(\mf p^{k}))+O(\vert h(\mf p^{k})\vert^{2})\Bigr),$$
where the first equality follows from the unique factorization of ideals and the fact that $\mf p\vert Qu+a\Rightarrow \NN(\mf p)\geq M$, and $p^{k}\Vert Qu+a$ means that $k\in\N$ is the largest natural number for which $Qu+a\in{\mf p}^{k}$.
By the inequality $\left|\prod_{i=1}^{k}z_{i}-\prod_{i=1}^{k}w_{i}\right|\leq\sum_{i=1}^{k}\vert z_{i}-w_{i}\vert$ for  $z_{i},w_{i}\in\U$, which can be proved by induction, we have that
\begin{equation}\label{ffd}
    \begin{split}
       f'(Qu+a)=\exp(h(Qu+a))+O\left(\sum_{\mf p^{k}\Vert Qu+a, \NN(\mf p)>M}\vert h(\mf p^{k})\vert^{2}\right).
    \end{split}
\end{equation}
Also
note that
\begin{equation}\label{ddff}
    \begin{split}
        &\quad F(\tilde{f'},M,N;Q)=\sum_{M<\NN(\mf p^{k})\leq N, \mf p \nmid Q}h(\mf p^{k})\NN(\mf p)^{-k}(1-\NN(\mf p)^{-1})+O(M^{-1})\\&\quad\quad\quad\quad\quad\quad\quad\;:=\mu_{h,M,N}+O(M^{-1}).
    \end{split}
\end{equation}
Indeed, the term with $k=1$ can be compared to the actual definition of $F(\tilde{f'},M,N;Q)$, giving us an error bounded above by 
\begin{equation}\label{ddff3}
    \begin{split}
        \sum_{M<\NN(\mf p)\leq N} \sum_{k=2}^{\infty}\NN(\mf p)^{-k}=\sum_{M<\NN(\mf p)\leq N} \frac{O(1)}{\NN(p)^{2}}=O(1/M).
    \end{split}
\end{equation}
%

Combining (\ref{ffd}), (\ref{ddff}) together with the fact that 
  $\vert e^{w}-e^{z}\vert\ll \vert w-z\vert$, we have 
\begin{equation}\label{firstO1Kbound}
    \begin{split}
        &\quad\E_{\NN(u) \leq N}\vert f'(Qu+a)-\exp(F(\tilde{f'},M,N;Q))\vert
        \\&\ll \E_{\NN(u) \leq N}\vert \exp(h(Qu+a))-\exp(F(\tilde{f'},M,N;Q))\vert+\E_{\NN(u) \leq N}\sum_{\mf p^{k}\Vert Qu+a, \NN(\mf p)>M}\vert h(\mf p^{k})\vert^{2}
        \\&\ll \E_{\NN(u) \leq N}\vert h(Qu+a)-F(\tilde{f'},M,N;Q)\vert+\E_{\NN(u) \leq N}\sum_{\mf p^{k}\Vert Qu+a, \NN(\mf p)>M}\vert h(\mf p^{k})\vert^{2}
        \\&=\E_{\NN(u) \leq N}\vert h(Qu+a)-\mu_{h,M,N}\vert+\E_{\NN(u) \leq N}\sum_{\mf p^{k}\Vert Qu+a, \NN(\mf p)>M}\vert h(\mf p^{k})\vert^{2}+O(M^{-1}).
    \end{split}
\end{equation}

For the second term, we have
\begin{equation}\label{204}
    \begin{split}
        &\quad \E_{\NN(u) \leq N}\sum_{\mf p^{k}\Vert Qu+a, \NN(\mf p)>M}\vert h(\mf p^{k})\vert^{2}
        =\sum_{\NN(\mf p)>M, k\geq 1}\vert h(\mf p^{k})\vert^{2}\frac{1}{\gamma_{K} N}\sum_{\NN(u) \leq N}\mathds{1}_{\mf p^{k}\Vert Qu+a}
        \\&=\sum_{\NN(\mf p)>M, k\geq 1}\vert h(\mf p^{k})\vert^{2}\NN(\mf p^{k})^{-1}+o_{N\to\infty}(1) \text{ (by Lemma \ref{countingCRT})}.
    \end{split}
\end{equation}


For the first term, by Theorem \ref{TK},
we have 
\begin{equation}\label{20}
    \begin{split}
        &\quad \E_{\NN(u) \leq N}\vert h(Qu+a)-\mu_{h,M,N}\vert 
        \\&=\E_{\NN(u) \leq N}\left| h(Qu+a)-\sum_{M<\NN(\mf p^{k})\leq N, \mf p\nmid Q}h(\mf p^{k})\NN(\mf p)^{-k}(1-\NN(\mf p)^{-1})\right|
        \\&\leq \left(\E_{\NN(u) \leq N}\left| h(Qu+a)-\sum_{M<\NN(\mf p^{k})\leq N,\mf p\nmid Q}h(\mf p^{k})\NN(\mf p)^{-k}(1-\NN(\mf p)^{-1})\right|^{2}\right)^{1/2}
        \\&\leq \left((2+o_{M;N\to\infty}(1))\cdot\sum_{\NN(\mf p)>M, k\geq 1}\vert h(\mf p^{k})\vert^{2}\NN(\mf p^{k})^{-1}+\right. 
        \\&\quad\quad\quad\quad\quad\quad\quad\quad\quad\;\quad \left. O(1)\cdot \frac{1}{\sqrt{N}}\sum_{\NN(\mf p^{k})\leq N', \NN(\mf p)>M}h(\mf p^{k})^{2}\cdot \NN(\mf p^{k})^{-1/2}\right)^{1/2}.
    \end{split}
\end{equation}



Note that by the Cauchy-Schwarz inequality and Parts (i) and (ii) of Lemma~\ref{l2},
\begin{equation}\label{208}
    \begin{split}
      &\qquad\frac{1}{\sqrt{N}}\sum_{\NN(\mf p^{k})\leq N', \NN(\mf p)>M}\vert h(\mf p^{k})\vert^{2}\cdot \NN(\mf p^{k})^{-1/2}
    \leq \frac{4}{\sqrt{N}}\sum_{\NN(\mf p^{k})\leq N'} \NN(\mf p^{k})^{-1/2}
       \\ & \quad\quad\quad\quad\quad\quad\quad\quad\quad\quad\quad\leq \frac{4}{\sqrt{N}}\left(\sum_{\NN(\mf p^{k})\leq N'} \NN(\mf p^{k})^{-1}\right)^{1/2}\cdot \left(\sum_{\NN(\mf p^{k})\leq N'} 1\right)^{1/2}
       \\&\quad\quad\quad\quad\quad\quad\quad\quad\quad\quad\;\;\;=O\left(\frac{\log\log N'}{\log N'}\right)^{1/2}=o_{M;N\to\infty}(1).
    \end{split}
\end{equation}
 So now it remains to control $\sum_{\NN(\mf p)>M, k\geq 1}\vert h(\mf p^{k})\vert^{2}\NN(\mf p^{k})^{-1}$. Similarly to (\ref{ddff3}), the trivial bound for $|h(\mf p^k)|$ gives
\begin{equation}\label{201}
    \begin{split}
       \sum_{\NN(\mf p)>M, k\geq 2}\vert h(\mf p^{k})\vert^{2}\NN(\mf p^{k})^{-1}=O(M^{-1}).
    \end{split}
\end{equation}
 On the other hand,
\begin{equation}\label{202}
    \begin{split}
\sum_{\NN(\mf p)>M}\vert h(\mf p)\vert^{2}\NN(\mf p)^{-1}=\sum_{\NN(\mf p)>M}2(1-\text{Re}\tilde{f'}( \mf p))\NN(\mf p)^{-1}=2\D(\tilde{f'},1;M,\infty),    
\end{split}
\end{equation}
The conclusion follows by
combining (\ref{firstO1Kbound}), (\ref{204}), (\ref{20}), (\ref{208}), (\ref{201}) and (\ref{202}).
\end{proof}

\section{Proofs of the recurrence results}\label{sec: proofmain}

In this section, we prove Propositions \ref{aperiodiclimit0},
\ref{zeroaveragesadeltabdelta}, and
\ref{fffss}.
  As we saw in Section~\ref{sec: proofstrat}, they will imply the partition regularity results promised in Section \ref{s1}.

Throughout this section, we assume that $K=\Q(\sqrt{-d})$ with $d\in\N$ squarefree.
We will make use of the next two simple lemmas:
\begin{lemma}\label{multav0}
Let $f: \mathcal{O}_K^{\times} \to \U$ be a non-trivial multiplicative function (that is, $f \neq 1$). Let $\Phi_M$ be a multiplicative F\o lner sequence in $\mathcal{O}_K^{\times}$. Then,
\[
\lim_{M \to \infty} \E_{u \in \Phi_M} f(u) = 0.
\]
\end{lemma}
\begin{proof}
Since $\U$ is compact, we consider limit points of the sequence $(\E_{u \in \Phi_M}f(u))_{M \in \N}$. Abusing notation, suppose that the limit of this sequence exists. 

Since $f\neq 1$, let $v \in \mathcal{O}_K^{\times}$ be such that $f(v)\neq 1$. Since $\Phi_M$ is a F\o lner sequence, we can write
\[
\lim_{M \to \infty} \E_{u \in \Phi_M} f(u) = \lim_{K \to \infty} \E_{u \in v\Phi_M} f(u) = \lim_{M \to \infty} \E_{u \in \Phi_M} f(vu) = f(v) \lim_{M \to \infty} \E_{u \in \Phi_M} f(u).
\]
Since $f(v) \neq 1$, this implies that the limit in question must be $0$. Therefore, $0$ is the only possible accumulation point for this sequence, so compactness of $\U$ completes the proof.
\end{proof}
\begin{lemma}\label{sequenceavfact}
Let $d\in\N$ be squarefree and $a: \OK^{\times} \to \U$ be a sequence. Let $l_1, l_2 \in \OK$ not both $0$. Suppose that for some $\varepsilon>0$ and some sequence $b: \N \to \U$ we have
\[
\limsup_{N \to \infty} \E_{\NN(u) \leq N} |a(u)-b(N)| \leq \varepsilon.
\]
Then,
\[
\limsup_{N \to \infty} \E_{\NN(n), \NN(m) \leq N} |a(l_1m+l_2n)-b(lN)| \ll_{K,l} \varepsilon,
\]
where $l:=2(\NN(l_1)+\NN(l_2))$.
\end{lemma}
\begin{proof}
We can estimate
\begin{equation}\label{sequenceavfactE1}
\E_{\NN(n), \NN(m) \leq N} |a(l_1m+l_2n)-b(lN)| \ll_{K,l} \frac{1}{N^2}\sum_{\NN(u) \leq lN} w_N(u) |a(u)-b(lN)|,
\end{equation}
%
%
%
where for $u \in \OK^{\times}$ we put $$w_N(u):=\left|\{(m,n)\in (\OK)^{2}: \NN(n), \NN(m) \leq N \text{ and }l_1m+l_2n=u\}\right|.$$ Since, given an $m$ with $\NN(m) \leq N$, there is only one solution to $l_1m+l_2n=u$ (as an equation in $n$), we see that $w_N(u) \ll_{K} N$. On the other hand, it is not hard to see that for the $u\in\OK$ for which $w_{N}(u)$ is non-empty has size at most $lN$. So the right hand side of \eqref{sequenceavfactE1} is bounded above   by
\[
\ll_{K,l}\E_{\NN(u) \leq lN} |a(u)-b(lN)|,
\]
so we are done, given our starting hypothesis.
\end{proof}
Before moving onto the main proof, we need a couple of preliminary results. Consider the sets on $\S^1$ given by
\[
I(\delta):=\{\exp(2\pi i\phi) : \phi \in (-\delta,\delta)\}.
\]

\begin{lemma}\label{weights}
Let $0<\delta<\frac{1}{2}$ and consider the trapezoidal function $F_{\delta}: \S^1 \to [0,1]$ which is equal to $1$ on $I_{\delta/2}$ and $0$ outside of $I_{\delta}$. For $\ell,\ell' \in \mathcal{O}_K^{\times}$ and $\alpha,\beta \in \OK$, let
\begin{equation}\label{weight1}
w_{\delta}(m,n):= F_{\delta}\left(\NN\Bigl(\frac{\ell (m+\alpha n)(m+\beta n)}{\ell' mn}\Bigr)^i \right)\cdot \mathds{1}_{S_q}(m,n), \quad m, n \in \mathcal{O}_K^{\times}.
\end{equation}
Then,
\begin{equation}\label{weight2}
\lim_{N \to \infty} \E_{\NN(m), \NN(n) \leq N } w_{\delta}(m,n)>0.
\end{equation}
\end{lemma}
\begin{proof}
The left hand side of (\ref{weight2}) can be rewritten as 
\begin{equation}\label{weight3}
\lim_{N \to \infty} \E_{m',n'\in\frac{1}{N}\OK\colon\NN(m'), \NN(n') \leq 1}F_{\delta}\left(\NN\Bigl(\frac{\ell (m'+\alpha n')(m'+\beta n')}{\ell' m'n'}\Bigr)^i \right) \mathds{1}_{S_q}(m',n').
\end{equation}

We extend the domain of $\NN$ from $K$ to $\C$ by continuity, since $K$ is dense in $\C$.
Let $B:=\{ w \in \C : \NN(w) \leq 1\}$. Define
  $$G_{\delta}(w,z):=F_{\delta}\left(\NN\Bigl(\frac{\ell (w+\alpha z)(w+\beta z)}{\ell' wz}\Bigr)^i \right)\cdot\mathds{1}_{B\times B, (w,z) \in S_q}(w,z).$$ 
  Then $G_{\delta}$ is continuous except for a set of measure zero with respect to the Lebesgue measure on $B \times B$, where $B:=\{ w \in \C : \NN(w) \leq 1\}$ -- here $(w,z) \in S_q$ is checked with the same definition, which also makes sense for $w,z \in \C$. Reinterpreting (\ref{weight3}) as Riemann sums, we may rewrite (\ref{weight3}) as   
\[
\int_{B}\int_{B} G_{\delta}(w,z) \ dwdz,
\]
where $dwdz$ is the appropriately normalized Lebesgue measure on $B \times B$. 

It remains to show that the integral of $G_{\delta}$ over $B \times B$ is positive. Since $G_{\delta}$ is non-negative, it suffices to show that $G_{\delta}$ is not identically $0$ outside the set of points where either $z$ or $w$ is equal to $0$; more particularly, we will show it is non-zero along a line of the form $w=az$, for some real number $a>1$. It follows that $G_{\delta}$ is non-zero for a small tube around said line (since $a>1$ implies that $w,z$ are distinct and non-zero), which implies the asserted positivity. 

In order to see that this is the case, observe that if we set $w=az$ for $a \in \R$ and $k \in \N$, the property $\NN(\ell(w+\alpha z)(w+\beta z))=\NN(\ell'wz)e^{2k\pi}$ will be satisfied, provided that we can solve the equation 
\[
\NN(\ell(a+\alpha)(a+\beta))=\NN(\ell'a)e^{2k\pi}.
\]
It is easy to see that if we move $\NN(a)$ to the left hand side of this equation, the function $H(a):=\NN((1+\beta a^{-1})(a+\alpha))$ is well defined for $a>0$, and because it is increasing, its range contains a ray of the form $[x_0,\infty)$. Now, simply pick $k \in \N$ such that $\frac{\NN(\ell')}{\NN(\ell)}e^{2k\pi}$ falls in this ray.

After raising such choices of norms of $w$ and $z$ to the $i$ and $-i$ powers respectively, we find that the function $F_{\delta}$ takes on values that are very close to $1$. Thus, by continuity, all points on a tubular neighbhorhood around it, will have value at least $1/2$, and also be non-zero, and such that $w \neq z$. This implies, as we argued, that $\int_B\int_B G_{\delta} \ dwdz>0$, as desired.



\end{proof}

\begin{proof}[Proof of Proposition \ref{aperiodiclimit0}]
By the definition of $w_{\delta}(m,n)$, we can use the Stone-Weierstrass theorem on $\S^1$, so that together with linearity of the limit, the continuous function $F_{\delta}$ appearing in the weights can be replaced by a power $z^k$. This will simplify matters considerably, transforming the averages we need to show converge to $0$ into 
\[
\lim_{N \to \infty} \E_{\NN(m), \NN(n) \leq N} \mathds{1}_{S_q}(m,n)\NN\Bigl(\frac{\ell(m+\alpha n)(m+\beta n)}{\ell' mn}\Bigr)^{ki}  \times\]
\[
f(\ell(Qm+1+\alpha Qn)(Qm+1+\beta Qn))\cdot \overline{f(\ell'(Qm+1)(Qn))},
\]
for $k \in \Z$. Note that, with the same trick we used in Lemma~\ref{weights} we may as well replace $n, m$ by $Qn$ and $Qm$ respectively inside the norms, because the changes cancel out with the complex conjugate.

We next argue that, in fact, we can change $Qm$ by $Qm+1$ without affecting the averages. Indeed, for given $n \in \N$, it is not hard to see that 
$$\lim_{\NN(m) \to \infty} \log \frac{\NN(\ell(Qm+Q\alpha n)(Qm+Q\beta n))}{\NN(\ell(Qm+1+\alpha Qn)(Qm+1+\beta Qn))}=\lim_{\NN(m) \to \infty} \log \frac{\NN(\ell'(Q^2mn))}{\NN(\ell'(Qm+1)Qn)}=0.$$ So this change is valid. 

Therefore, it is enough to establish the convergence to $0$ of the averages
\[
\E_{\NN(m), \NN(n) \leq N} \mathds{1}_{S_q}(m,n)f_k(\ell(Qm+1+\alpha Qn)(Qm+1+\beta Qn))\cdot \overline{f_k(\ell'(Qm+1)(Qn))},
\]
where $f_k(n):=f(n) \cdot \NN(n)^{ik}$.
Since $f$ is aperiodic, so are $f_k$ and $\overline{f_k}$ by Theorem \ref{characterization}. So the conclusion follows from Proposition~\ref{P: appendix}. 
%
%
%
\end{proof}

\begin{proof}[Proof of Proposition \ref{zeroaveragesadeltabdelta}]
By hypothesis on $f$,  $f(\e)=1$ for all units $\e$ and there exist some modified Dirichlet character $\chi$ of period $I$, some $\tau\in\R$, and some extension $\tilde{f'}$ of the function
$$f'(u):=f(u)\overline{\chi}(u)\NN(u)^{-i\tau}$$
such that 
  $\D(\tilde{f'}, 1)<\infty$.
Fix $\varepsilon>0$ and take $M_{0}=M_{0}(\varepsilon,f)$ so that 
\begin{equation}\label{tail}
    \sum_{\NN(\mf p)>M_{0}}\frac{1}{\NN( \mf p)}(1-\text{Re}\tilde{f'}(\mf p))+M_{0}^{-1/2}\leq \varepsilon
\end{equation}
and that the set $\Phi_{M_{0}}$ defined in (\ref{folnerMult}) is contained in $I$.  
By Proposition \ref{p25} and (\ref{tail}), 
$$\limsup_{N\to\infty}\E_{\NN(u)\leq N}\vert f(Qu+1)-\NN(Qu)^{i\tau}\exp(F(\tilde{f'},M,N;Q))\vert\ll\varepsilon$$
for every $M>M_{0}$ and $Q\in\Phi_{M}$.
For fixed  $\alpha \in \mathcal{O}_K$, it follows from Lemma~\ref{sequenceavfact}  that
\[
\limsup_{N\to\infty}\E_{\NN(m), \NN(n) \leq N}\left| f(Q(m+ \alpha n))+1)-\NN(Q(m+\alpha n))^{i\tau}\exp(F(\tilde{f'},M,lN;Q))\right|\ll\varepsilon,
\]
where $l_{\alpha}:=2(\NN(\alpha)+1)$.
Combining the above two estimates together with the triangle inequality, 
we have that
 \begin{equation}\label{fweqfwf} 
    \begin{split}
&\limsup_{N\to\infty}\E_{\NN(m),\NN(n)\leq N}\Bigl\vert f\Bigl(\frac{(Qm+1+Q\alpha n)(Qm+1+Q\beta n)}{Qm+1}\Bigr) 
       \\&-\NN\Bigl(\frac{(Qm+Q\alpha n)(Qm+Q\beta n)}{Qm}\Bigr)^{i\tau}\exp(F(\tilde{f'},M,l_{\alpha}N;Q)+F(\tilde{f'},M,l_{\beta}N;Q)\\&\quad\quad\quad\quad\quad\quad\quad\quad\quad\quad\quad\quad\quad\quad\quad\quad\quad\quad\quad\quad\quad\quad\quad\quad\quad\quad\quad-F(\tilde{f'},M,N;Q))\Bigr\vert\ll\varepsilon. 
    \end{split}
\end{equation}
We now multiply both sides of (\ref{fweqfwf}) by $w_{\delta}(m,n)\cdot f(\frac{\ell}{\ell'n})\NN(Q)^{i\tau}$, which is bounded in absolute value by $1$, to obtain
 \begin{equation}\label{fpwepfw} 
    \begin{split}
       &\limsup_{N\to\infty}\E_{\NN(m), \NN(n) \leq N}  
         \Bigl\vert w_{\delta}(m,n)\cdot f\Bigl(\frac{\ell(Qm+1+Q\alpha n)(Qm+1+Q\beta n)}{\ell'(Qm+1)Qn}\Bigr)(f(Q)\NN(Q)^{-i\tau})
       \\&-w_{\delta}(m,n)\cdot f(\frac{\ell}{\ell'n}) \NN\left(\frac{(m+\alpha n)(m+\beta n)}{mn}\right)^{i\tau} 
       \cdot\exp(F(\tilde{f'},M,l_{\alpha}N;Q)+F(\tilde{f'},M,l_{\beta}N;Q)\\&\quad\quad\quad\quad\quad\quad\quad\quad\quad\quad\quad\quad\quad\quad\quad\quad\quad\quad\quad\quad\quad\quad\quad\quad\quad\quad\quad\quad\; -F(\tilde{f'},M,N;Q))\Bigr\vert\ll\varepsilon. 
    \end{split}
\end{equation}
Let
    $$\tilde{L}_{\delta}(Q):=\lim_{N \to \infty} \E_{\NN(m), \NN(n) \leq N}  w_{\delta}(m,n) \cdot f\Bigl(\frac{\ell(Qm+1+Q\alpha n)(Qm+1+Q\beta n)}{\ell'(Qm+1)Qn}\Bigr)(f(Q)\NN(Q)^{-i\tau}).$$
    Note that if $Q\in\Phi_{M}$, then $F(\tilde{f'},M,N;Q)$ is independent of $Q$ for all $N$. So  (\ref{fpwepfw}) implies that $|\tilde{L}_{\delta}(Q)-\tilde{L}_{\delta}(Q')|\ll \varepsilon$  for all $Q, Q' \in \Phi_M$. 
%
%
%
%
Therefore, letting $M\to \infty$ we obtain
$$\limsup_{M\to\infty}\max_{Q,Q'\in\Phi_{M}}\vert\tilde{L}_{\delta}(Q)-\tilde{L}_{\delta}(Q')\vert=0.$$
For each $M\in\N$, pick some $Q_{M}\in\Phi_{M}$. Then
$$\lim_{M \to \infty} \E_{Q \in \Phi_M} \lim_{N \to \infty} \E_{\NN(n),\NN(m) \leq N} A_{\delta}(f,Q; m,n)=\limsup_{M\to\infty}\tilde{L}_{\delta}(Q_{M})\E_{Q\in\Phi_M}\overline{f}(Q)\NN(Q)^{i\tau}.$$
Since $Q\mapsto \overline{f}(Q)Q^{i\tau}$ is a nontrivial multiplicative function, it follows from Lemma~\ref{multav0} that $\lim_{M \to \infty} \E_{Q\in\Phi_{M}}\overline{f}(Q)\NN(Q)^{i\tau}=0$ and we are done.

\end{proof}

\begin{proof}[Proof of Proposition \ref{fffss}]


Let $a:=\sigma(\{1\})>0$ and for $\delta>0$ set
    $$\mu_{\delta}:=\lim_{N\to\infty}\E_{\NN(m), \NN(n) \leq N}w_{\delta}(m,n).$$
By Lemma~\ref{weights}, we have that the limit exists and $\mu_{\delta}>0$. For $T>0$, denote
$$\mathcal{A}_{T}:=\{(\NN(u)^{i\tau})_{u\in\mathcal O_K^{\times}}\colon \tau\in [-T,T]\}\subseteq \mathcal{A},$$
which we claim is a closed Borel set. Indeed, we can write
\[
\mathcal{A}_T= \bigcap_{u \in \mathcal O_K^{\times}} \pi_u^{-1}\left(\{ \NN(u)^{i\tau} : \tau \in [-T,T] \}\right),
\]
where $\pi_u$ is the projection onto the $u$-th coordinate, with $u \in \mathcal O_K^{\times}$. Now, given $\tau \in \R$, the set $\{\NN(u)^{i\tau} : \tau \in [-T,T] \}$ is the image of the compact set $[-T,T]$ under the continuous map $\tau \mapsto\NN(u)^{-i\tau}$, which means it is compact, and since the target space is Hausdorff, it must be closed. The topology taken on $\mathcal{M}$ makes each of the $\pi_u$ continuous, and an aribtrary intersection of closed sets is closed. Thus, $\mathcal{A}_T$ is a Borel set.

By the monotone convergence theorem for sets, there exists $T_{0}=T_{0}(\sigma)>0$ such that $\sigma(\mathcal{A}\backslash\mathcal{A}_{T_{0}})\leq a/2$.
Since $\lim_{\NN(u)\to\infty}\sup_{Q\in \mathcal O_K^{\times}}\vert\log\NN(Qu+1)-\log\NN(Qu)\vert=0$,
we have that 
\begin{equation}\label{e78}
    \begin{split}
        &\lim_{N\to\infty}\sup_{f\in \mathcal{A}_{T_{0}},Q\in\mathcal O_K^{\times}}\E_{\NN(m), \NN(n) \leq N}w_{\delta}(m,n)\cdot\Bigl\vert f\Bigl(\frac{\ell(Qm+1+\alpha Qn)(Qm+1+\beta Qn)}{\ell'(Qm+1)Qn}\Bigr) 
        \\&\quad\quad\quad\quad\quad\quad\quad\quad\quad\quad\quad\quad\quad\quad\quad\quad\quad\quad\quad\quad\quad\quad\;\;\;\;-f\Bigl(\frac{\ell(m+\alpha n)(m+\beta n)}{\ell'mn}\Bigr)\Bigr\vert=0.
    \end{split}
\end{equation}
On the other hand, by the definition of $w_{\delta}$ above, the limit in $N$ after the supremum will only have contributions close to the value of $w_{\delta}$ when the value of $f\Bigl(\frac{\ell(m+\alpha n)(m+\beta n)}{\ell'mn}\Bigr)$ is very close to $1$. This will be independent of the choice of multiplicative function $f$ provided it is in the set $\mathcal{A}_{T_0}$, so it follows that
\begin{equation}\label{e79}
    \begin{split}
        \lim_{N\to\infty}\sup_{f\in \mathcal{A}_{T_{0}}} \left| \E_{\NN(m), \NN(n) \leq N}w_{\delta}(m,n)\cdot\Bigl( f\Bigl(\frac{\ell(m+\alpha n)(m+\beta n)}{\ell'mn}\Bigr)-1\Bigr)\right|=O(\delta).
    \end{split}
\end{equation}
 Combining (\ref{e78}) and (\ref{e79}), we have that 
Thus, taking the last limit in $\delta$ yields the equality that was claimed.
    \begin{equation}
    \begin{split}
       \lim_{\delta\to 0^{+}}\lim_{N\to\infty}\sup_{f\in \mathcal{A}_{T_{0}}} \left| \E_{\NN(m), \NN(n) \leq N}w_{\delta}(m,n)\cdot\Bigl( f\Bigl(\frac{\ell(Qm+1+\alpha Qn)(Qm+1+\beta Qn)}{\ell'(Qm+1)Qn}\Bigr)-1\Bigr)\right|=0.
    \end{split}
\end{equation}
So if $\delta_{0}$ is sufficiently small depending only on $T_{0}$, we have that
$$\liminf_{N\to\infty}\inf_{Q\in\mathcal O_K^{\times}}\text{Re}\left(\E_{\NN(m), \NN(n) \leq N}\int_{\mathcal{A}_{T_{0}}}w_{\delta}(m,n)f\Bigl(\frac{\ell(Qm+1+\alpha Qn)(Qm+1+\beta Qn)}{\ell'(Qm+1)Qn}\Bigr)\,d\sigma(f)\right)$$
is at least $\sigma(\mathcal{A}_{T_{0}})\mu_{\delta_{0}}\geq \sigma(\{1\})\mu_{\delta_{0}}=a\mu_{\delta_{0}}$.
On the other hand, 
$$\liminf_{N\to\infty}\sup_{Q\in\mathcal O_K^{\times}}\left|\E_{\NN(m), \NN(n) \leq N}\int_{\mathcal{A}\backslash\mathcal{A}_{T_{0}}}w_{\delta}(m,n)f\Bigl(\frac{\ell(Qm+1+\alpha Qn)(Qm+1+\beta Qn)}{\ell'(Qm+1)Qn}\Bigr)\,d\sigma(f)\right|$$
is bounded above by $a\mu_{\delta_{0}}/2$.  We are done by setting $\rho_{0}:=a\mu_{\delta_{0}}/2$.
\end{proof}

\subsection{Further discussions}\label{SEC: further}
We end this article with a discussion of potential directions that can be studied next. More particularly, we will discuss the case of real quadratic fields, and also comment on the potential to develop quadratic concentration estimates (as a generalization of the linear concentration estimates that were obtained in Section~\ref{SECconcest}).

\subsubsection{Real quadratic fields}
The main issue with real quadratic fields is the fact that they have infinitely many units. This breaks down many of our arguments irreparably, especially because the balls one could hope to define with the norm (or even its absolute value) have infinitely many elements. This means that the connection with the generalized version of Hal\'asz's theorem, one of the key ingredients of this work, cannot be made.

A potential approach one may consider is to average along a suitable cut of the ball, where we only take a few representatives of each element times powers of the units. Next, it would be key to be able to apply a suitable version of Hal\'asz's theorem in this context, switching to boxes instead of balls, if necessary (and then showing that the two averaging methods can be compared). It might by useful to start by assuming that $f(\e)=1$ for all units $\e$, and then upgrade to an arbitrary $f$.

\subsubsection{Quadratic concentration estimates for rings of integers}
We conjecture that it is possible to develop quadratic concentration estimates in the spirit of \cite{FKM2}, but due to space considerations we have not attempted to develop them in the present paper. As a first step, one could try to prove them for $\Z[i]$ given that it is the closest to $\Z$, having a norm that is easy to work with, and still preserving the UFD property. Then, one can see if they can be extended to arbitrary quadratic imaginary fields, and from there try to move to general number fields (of course, the issue we raised in the previous subsubsection above of the existence of infinitely many units still needs to be circumvented).

\appendix
\section{Seminorm estimates}\label{appapp}
The purpose of this appendix is to deal with the apparent disconnect between Halász's theorem (Theorem~\ref{C22intro}), which deals with averages along balls, and a result that is incredibly important to us: \cite[Theorem~1.12]{S23}, which works for averages along boxes, instead of balls.
%
%
%
To this end, we prove the following.
\begin{proposition}\label{P: appendix}
Let $K=\Q(\sqrt{-d})$ for some squarefree $d\in\N$, $s \in \N$ with $s \geq 3$ and $L_j(m,n)$, $j=1,\dots, s$ be linear forms with coefficients in $\OK$ such that the linear forms $L_1, L_j$ are linearly independent for $j=2,\dots, s$. For $j=1,\dots, s$, let $h_j: \OK \to \C$. If  $h_{1}$ is aperiodic, 
then
\begin{equation}\nonumber
   \lim_{N \to \infty} \E_{\NN(m),\NN(n) \leq  N} \prod_{j=1}^s h_j(L_j(m,n))  = 0.
\end{equation}
 \end{proposition}
The first step to proving Proposition \ref{P: appendix} is to bound the target average by the Gower's norm of a modified version of $h_{1}$. 
Clearly there exists $C_{1},C_{2}>0$ depending only on $K,L_{1},\dots,L_{s}$ such that
$$\NN(m),\NN(n)\leq N\Rightarrow \NN(L_{j}(m,n)) \leq  C_{1}N\Rightarrow m,n\in[\pm C_{2}\sqrt{N}]^{2},$$
where with a slight abuse of notations, we think of numbers $m,n\in\OK$ as a vector in $\Z^{2}$ (with $1,\tau_{d}$ being the basis).
Let $\tilde{N}$ be the smallest prime number which is greater than $10C_{2}\sqrt{N}$.
Let $h_{j,N}\colon \left[\pm\frac{\tilde{N}-1}{2}\right]^{2}\to\C$ be the map given by $h_{j,N}(u):=\mathds{1}_{\NN(u)\leq C_{1}N}\cdot h_{j}(u)$.
Then, for every $N \in \N$, 
\begin{equation}\label{fwewgb}
    \E_{\NN(m),\NN(n) \leq  N} \prod_{j=1}^s h_j(L_j(m,n))=W_N\cdot \E_{m,n \in  \left[\pm\frac{\tilde{N}-1}{2}\right]^{2}} \mathds{1}_{\NN(m),\NN(n)\leq N}\prod_{j=1}^s h_{j,N}(L_j(m,n)),
\end{equation}
where 
$$W_N:=\frac{\tilde{N}^{4}}{\vert\{(m,n)\in(\OK)^{2}\colon \NN(m),\NN(n) \leq  N\}\vert}$$
is a finite positive real number. 
%
%
By Bertrand's postulate and Corollary \ref{lem: bounds for number of ideals of norm at most x},
  there exist  $0<K_1<K_2$ depending only on $K,L_{1},\dots,L_{s}$ such that
 \begin{equation}\label{k1k2}
    0<K_1 \leq \liminf_{N \to \infty} W_N \leq \limsup_{N \to \infty } W_N \leq K_2 <\infty.
    \end{equation}


Let $\tilde{h}_{j,N}\colon \Z_{\tilde{N}}^{2}\to\C$ be the map given by $\tilde{h}_{j,N}:=h_{j,N}\circ\tau$, where $\tau\colon \Z_{\tilde{N}}^{2}\to \left[\pm\frac{\tilde{N}-1}{2}\right]^{2}$ is the natural embedding.
Then it follows from
 (\ref{fwewgb}) that
    \begin{equation}\label{eq: averagetobebdd}
    \E_{\NN(m),\NN(n) \leq  N} \prod_{j=1}^s h_j(L_j(m,n))=W_{N}\cdot\E_{m,n \in  \Z_{\tilde{N}}^{2}} \mathds{1}_{\NN(\tau(m)),\NN(\tau(n))\leq N}\prod_{j=1}^s \tilde{h}_{j,N}(\tilde{L}_j(m,n)),
    \end{equation}
    where $\tilde{L}_{j}\colon \Z_{\tilde{N}}^{4}\to \Z_{\tilde{N}}^{2}$ is the linear map induced by $L_{j}\colon \Z^{4}\to\Z^{2}$.

    
 By (\ref{k1k2}), it follows from \cite[Proposition 7.1$'$]{GT10b} that the limit of (\ref{eq: averagetobebdd}) as $N\to\infty$ is zero if $\lim_{N\to\infty}\norm{\tilde{h}_{1,N}}_{U^{s-1}(\Z_{\tilde{N}}^{2})}=0$, where $\norm{\cdot}_{U^{s-1}(\Z_{\tilde{N}}^{2})}$ is the $(s-1)$-th Gowers norm on $\Z_{\tilde{N}}^{2}$ (see for example \cite[Definition B.1]{GT10b} for the definition). 
    The reasons said Proposition 7.1' can be applied are as follows. First, the pseudorandomness assumption is unneeded in our case, since our functions are bounded by $1$. Second, without loss of generality, we can assume that our linear forms are in the $s$-normal form, because the operation as described in \cite[Lemma~4.4]{GT10b}.

To conclude the proof of Proposition~\ref{P: appendix}, it suffices to show $\lim_{N\to\infty}\norm{\tilde{h}_{1,N}}_{U^{s-1}(\Z_{\tilde{N}}^{2})}=0$.
A similar estimate was obtained in \cite{S23} where $h_{j,N}(u)$ is replaced by a function of the form (say) $\mathds{1}_{u\in[\pm \sqrt{N}]^{2}}\cdot h_{j}(u)$ (see \cite[Theorem~1.12]{S23}). Our goal is to show that an argument similar to the ones used in Sections 7 and 8 of \cite{S23} can be applied to prove Proposition \ref{P: appendix}.


In the rest of the proof we refer the reader to \cite{S23} for definitions.
Suppose on the countary that there exists an arbitraily large $N$ such that 
$\norm{\tilde{h}_{1,N}}_{U^{s-1}(\Z_{\tilde{N}}^{2})}>\e$. 
Then by the inverse theorem for Gowers norms (see for example \cite[Theorem~8.5]{S23}), there exist $\delta\gg_{\e,s} 1$ and an $(s-2)$-step $\tilde{N}$-periodic nilsequence $\phi_{N}\colon \Z_{\tilde{N}}\to\C$ of complexity $O_{\e,s}(1)$ such that 
$$\E_{u\in\Z_{\tilde{N}}^{2}}\tilde{h}_{1,N}(u)\phi_{N}(u)\gg \delta.$$
This is equivalent to saying that 
$$\E_{u\in\left[\pm\frac{\tilde{N}-1}{2}\right]^{2}}\mathds{1}_{\NN(u)\leq C_1N}h_{1}(u)\phi_{N}(u)\gg \delta.$$

Before continuing the rest of the proof, we  need to generalize a few properties and results from \cite{S23} that will be key for the rest of our arguments.

First, recall that for $P\subseteq \mathcal{O}_{K}$ and $p,q\in\mathcal{O}_{K}$, we denote by 
$$I_{p,q}(P):=\{u\in \mathcal{O}_{K}\colon pu,qu\in P\}.$$
Moreover, we say that a set $P\subseteq \mathcal{O}_{K}$ is \emph{good} if for every pair of primes $p,q\in\mathcal{O}_{K}$, every (finite length) 2-dimensional arithmetic progression $P'$, and every line $\ell$ in $\mathcal{O}_{K}$, the set $I_{p,q}(P\cap P')\cap \ell$ is a (one-dimensional) arithmetic progression.
 
We have the following generalization of \cite[Theorem~7.1]{S23}. 
\begin{proposition}\label{7.1s2}
    The analog of \cite[Theorem~7.1]{S23} holds if we replace $P$ by any good set (where the average is still taken along boxes).
\end{proposition}
\begin{proof}
The proof goes along the same lines as that of \cite[Theorem~7.1]{S23}. In order to avoid unnecessary repetition, we just wish to highlight the fact that the only key property of the set $P$ that we use is that it is a good set, that is: that it satisfies the intersection property with any line; everything else remains unchanged.    
\end{proof}

As a consequence, we have the following generalization of \cite[Theorem 8.1]{S23}.
\begin{proposition}\label{8.1s2}
    The analog of \cite[Theorem 8.1]{S23} holds if we replace $P$ by any good set, and $P'$ by a set which is the intersection of $P$ and a 2-dimensional infinite arithmetic progression (where the average is still taken along boxes).
\end{proposition}
\begin{proof}
    The proof also goes along the same lines as that of \cite[Theorem~8.1]{S23}. In order to avoid needlessly repeating it, we just wish to mention that the only key property of the set $P$ that we use is that it is a good set, that is: that it satisfies the intersection property with any line; everything else remains unchanged.
\end{proof}
In order to actually use these propositions, we need to make sure that the balls we are working with are good sets, which is the content of the following lemma:
\begin{lemma}\label{7.1s23}
For every large enough $N \in \N$ and $C>0$, the set $B_{CN}:=\{u\colon\NN(u)\leq CN\}$ is good.
\end{lemma}
\begin{proof}
Given a pair of primes $p, q \in \OK$, a 2-dimensional arithmetic progression $P$, and a one-dimensional line $\ell \in \mathcal{O}$ we would like to show that the set 
\[
\{ u \in \OK : pu, qu \in B_{CN} \cap P\} \cap \ell
\]
is a one-dimensional arithmetic progression. Indeed, for quadratic imaginary extensions, multiplication always by a prime element always enlarges the norm, so since $\NN(pu)=\NN(p)\NN(u) \geq \NN(u)$, and the same for $qu$, it follows that the first condition of belonging to $B_{CN}$ is nothing but $u \in B_{cN}$ for some other constant $c>0$ which will depend on the norms of $p$ and $q$. For $B_{cN}$ to be non-trivial, we need $N$ to be large enough. 

On the other hand, using the characterization for arithmetic progressions (only for quadratic imaginary fields) coming from the last part of Section~\ref{ss5} we see that, $pu \in P$ if and only if $u \in P'$, some other arithmetic progression $P'$, and the same for $qu$. The intersection of these two arithmetic progressions is again an arithmetic progression, say $P'',$ so the statement becomes
\[
\{u \in \OK : u \in B_{cN} \cap P''\} \cap \ell
\]
is an arithmetic progression, which is now geometrically clear (the key properties are the convexity of the balls, and the fact that $P''$ is a 2D arithmetic progression).
\end{proof}
We now continue the proof. Since 
 $$\E_{u\in\left[\pm\frac{\tilde{N}-1}{2}\right]^{2}}\mathds{1}_{\NN(u)\leq C_1N}h_{1}(u)\phi_{N}(u)\gg \delta$$
 and since the set $B_{C_1N}=\{u\colon\NN(u)\leq C_1N\}$ is good, it follows from Proposition \ref{8.1s2} that 
 $$\left|\E_{u\in[\pm\frac{\tilde{N}-1}{2}]^{2}}\mathds{1}_{R_{N}}(u)h_{1}(u)\right|\gg_{\delta} 1$$
 for some $R_{N}$ which is the intersection of $B_{C_1N}$ and some infinite arithmetic progression $P_{N}$. Then 
  $$\left|\E_{u\in B_{N}}\mathds{1}_{R_{N}}(u)h_{1}(u)\right|\gg_{\delta} 1.$$
In conclusion, we have that 
$$\lim_{N\to\infty}\sup_{P\in\mathcal{AP}[\tau_{d}]}\left|\E_{\NN(u)\leq N}\mathds{1}_{P}(u)h_{1}(u)\right|\gg_{\delta} 1.$$
By Proposition \ref{prop: dircharequivalence}, this implies that $h_{1}$ is not aperiodic, a contradiction. This completes the proof of Proposition \ref{P: appendix}.

\bibliographystyle{abbrv}
\bibliography{library}

 
\end{document}